\documentclass{article}

\usepackage[T1]{fontenc}
\usepackage{enumerate, amsmath, amsfonts, amssymb, amsthm, dsfont, mathrsfs, wasysym, graphics, graphicx, xcolor, url, aliascnt, hyperref, hypcap, xargs, multicol, pdflscape, multirow, hvfloat, array, ae, aecompl, pifont, mathtools, a4wide, float, blkarray, overpic, nicefrac, stmaryrd, anyfontsize, yfonts, fontawesome}
\usepackage{pdflscape}
\usepackage{xargs, bbm, enumerate, paralist}
\usepackage[noabbrev,capitalise]{cleveref}
\usepackage[normalem]{ulem}
\usepackage{marginnote}
\usepackage{mathrsfs}
\usepackage{paralist}
\usepackage{multicol}
\usepackage{enumitem}
\usepackage{svg}
\usepackage{cleveref}
\usepackage[dvipsnames]{xcolor}

\hypersetup{colorlinks=true, citecolor=darkblue, linkcolor=darkblue}
\usepackage[all]{xy}
\usepackage{tikz}
\usepackage{tikz-cd}
\usepackage{tikz-3dplot}
\usetikzlibrary{trees, decorations, decorations.pathmorphing, decorations.markings, decorations.shapes, shapes, arrows, matrix, calc, fit, intersections, patterns, angles}
\graphicspath{{figures/}{figures/diagonals/}{figures/walks/}{figures/tubes/}{figures/blocks/}}
\makeatletter\def\input@path{{figures/}}\makeatother
\usepackage{caption}
\usepackage[export]{adjustbox}
\usepackage{tikzsymbols}
\usepackage{booktabs}
\usepackage{tabularx}

\usepackage{paralist}

\DeclareFontEncoding{LY}{}{}
\DeclareFontSubstitution{LY}{yfrak}{m}{n}
\DeclareFontEncoding{LYG}{}{}
\DeclareFontSubstitution{LYG}{ygoth}{m}{n}
\DeclareFontFamily{LYG}{ygoth}{}
\DeclareFontShape{LYG}{ygoth}{m}{n}{<->ygoth}{}
\DeclareFontFamily{LY}{yfrak}{}
\DeclareFontShape{LY}{yfrak}{m}{n}{<->yfrak}{}
\DeclareFontFamily{LY}{ysmfrak}{}
\DeclareFontShape{LY}{ysmfrak}{m}{n}{<->ysmfrak}{}
\DeclareFontFamily{LY}{yswab}{}
\DeclareFontShape{LY}{yswab}{m}{n}{<->yswab}{}

\theoremstyle{plain}

\newtheorem{theorem}{Theorem}[section]

\newaliascnt{corollary}{theorem}
\newtheorem{corollary}[corollary]{Corollary}
\aliascntresetthe{corollary}

\newaliascnt{proposition}{theorem}
\newtheorem{proposition}[proposition]{Proposition}
\aliascntresetthe{proposition}

\newaliascnt{lemma}{theorem}
\newtheorem{lemma}[lemma]{Lemma}
\aliascntresetthe{lemma}

\newaliascnt{conjecture}{theorem}

\aliascntresetthe{conjecture}

\newtheorem*{theorem*}{Theorem}

\theoremstyle{definition}

\newaliascnt{definition}{theorem}
\newtheorem{definition}[definition]{Definition}
\aliascntresetthe{definition}

\newaliascnt{example}{theorem}
\newtheorem{example}[example]{Example}
\aliascntresetthe{example}

\newaliascnt{remark}{theorem}
\newtheorem{remark}[remark]{Remark}
\aliascntresetthe{remark}

\newaliascnt{question}{theorem}
\newtheorem{question}[question]{Question}
\aliascntresetthe{question}

\newaliascnt{notation}{theorem}

\aliascntresetthe{notation}

\newaliascnt{assumption}{theorem}

\aliascntresetthe{assumption}

\newaliascnt{convention}{theorem}

\aliascntresetthe{convention}

\crefname{theorem}{theorem}{theorems}
\Crefname{theorem}{Theorem}{Theorems}

\crefname{corollary}{corollary}{corollaries}
\Crefname{corollary}{Corollary}{Corollaries}

\crefname{proposition}{proposition}{propositions}
\Crefname{proposition}{Proposition}{Propositions}

\crefname{lemma}{lemma}{lemmas}
\Crefname{lemma}{Lemma}{Lemmas}

\crefname{conjecture}{conjecture}{conjectures}
\Crefname{conjecture}{Conjecture}{Conjectures}

\crefname{definition}{definition}{definitions}
\Crefname{definition}{Definition}{Definitions}

\crefname{example}{example}{examples}
\Crefname{example}{Example}{Examples}

\crefname{remark}{remark}{remarks}
\Crefname{remark}{Remark}{Remarks}

\crefname{question}{question}{questions}
\Crefname{question}{Question}{Questions}

\crefname{notation}{notation}{notations}
\Crefname{notation}{Notation}{Notations}

\crefname{assumption}{assumption}{assumptions}
\Crefname{assumption}{Assumption}{Assumptions}

\crefname{convention}{convention}{conventions}
\Crefname{convention}{Convention}{Conventions}

\crefname{equation}{equation}{equations}
\Crefname{equation}{Equation}{Equations}

\newtheorem{maintheorem}{Theorem}

\crefname{maintheorem}{theorem}{theorems}
\Crefname{maintheorem}{Theorem}{Theorems}

\newcommand{\R}{\mathbb{R}} 
\newcommand{\N}{\mathbb{N}} 
\newcommand{\Z}{\mathbb{Z}} 
\renewcommand{\b}[1]{{\boldsymbol{#1}}} 

\renewcommand{\epsilon}{\varepsilon} 
\newcommand{\polytope}[1]{\mathsf{#1}}

\newcommand\pol{\polytope{P}}

\newcommand{\mathdefn}[1]{{\darkblue #1}} 

\DeclareMathOperator{\conv}{conv} 
\DeclareMathOperator{\nVol}{nVol}

\DeclareMathOperator{\relint}{relint}

\newcommand{\ie}{\textit{i.e.,}~} 
\definecolor{darkblue}{rgb}{0,0,0.7} 
\definecolor{green}{RGB}{57,181,74} 
\definecolor{violet}{RGB}{147,39,143} 
\newcommand{\darkblue}{\color{darkblue}} 
\newcommand{\defn}[1]{\textsl{\darkblue #1}} 

\newcommand{\Pyr}{\polytope{Pyr}}

\DeclareMathOperator{\aff}{aff}

\newcommandx{\Epoly}[3][1=\b s, 2=n, 3=t]{E_{#2}^{#1}(#3)}

\AtEndDocument{\bigskip{\footnotesize%
 \textsc{Universität Osnabrück, Germany} \\
Martina Juhnke: \texttt{martina.juhnke@uni-osnabrueck.de} \\ 
Justus Bruckamp \\
Jhon  B. Caicedo: \texttt{jhon.bladimir.caicedo.portilla@uni-osnabrueck.de}\par }}

\begin{document}

\title{Unimodular triangulations and Ehrhart theory for two families of Hermite normal form simplices}
\author{Justus Bruckamp, Jhon B. Caicedo and Martina Juhnke}
\date{}

\maketitle

\section*{Abstract}
We study regular unimodular triangulations, the integer decomposition property, and Ehrhart-theoretic properties of two families of Hermite normal form simplices.
We first consider the one-row case associated with the vector $(N - 1, \dots ,N - 1 , N)\in \N^d$, and completely characterize when the corresponding simplices admit a regular unimodular triangulation. 
Our constructions are explicit and also yield closed formulas for the $h^\ast$-polynomial and the local $h^\ast$-polynomial.
Moreover, we prove Ehrhart positivity and derive explicit dimension-dependent conditions under which the Ehrhart polynomial is not unimodal.

Finally, we extend our approach to the two-row cases associated with $(1, \dots ,1 , N)\in \N^d$ and $(M-1, \dots ,M-1, M, 0)\in \N^d$.
In these cases, we construct regular unimodular triangulations, derive closed formulas for the $h^\ast$-polynomial and the local $h^\ast$-polynomial, and prove Ehrhart positivity.

\section*{Acknowledgements}
Jhon B. Caicedo was supported by DFG grant JU 3097/4-1. 
He is grateful to Benjamin Braun and Andr\'es R. Vindas-Mel\'endez for their mathematical advice during the preparation of this work.

\section{Introduction}

The integer decomposition property, unimodular triangulations, and Ehrhart-theoretic properties are central topics in the theory of lattice polytopes. 
A lattice polytope admitting a unimodular triangulation satisfies the integer decomposition property (see \cite{Unimodular_triangulation_positive_result}), and when such a triangulation is shellable, the $h^\ast$-polynomial can be read directly from the combinatorics of the triangulation \cite{stanley1980decompositions}. 
Thus, explicit constructions of unimodular triangulations provide a useful bridge between discrete geometry, Ehrhart theory, and combinatorial commutative algebra.

In this paper, we study these questions for two classes of Hermite normal form simplices. 
Every lattice simplex is unimodularly equivalent to a simplex in Hermite normal form \cite{MR874114}, making this class a natural framework for explicit constructions and computations. 
We focus first on the one-row case. 
Although this family has been studied before (see \cite{Bajo2025LocalF, Hibi_Hermite}), only a few results are known concerning the integer decomposition property and unimodular triangulations for these simplices. 
Some of these results are formulated in terms of weighted projective space simplices; see \cite[Remark~1]{IDP_1} for  the  one-row Hermite normal form simplex.

For the purposes of this article, we focus on the family of one-row Hermite normal form simplex associated with vectors of the form $\b a=(N-1,\ldots,N-1,N)\in\N^d$.
This family was studied in \cite[Section~3.1]{Hibi_Hermite}, where the main objective was to study the shifted symmetry property of its $h^\ast$-vector. 
Our approach is different: we construct explicit triangulations and obtain the following characterization.
\begin{maintheorem}[\Cref{thm:Unimodular_triangulation},  \Cref{thm:not_IDP_2},  \Cref{thm:Not_IDP_1}]\label{thm:A}
Let $\polytope{S}_{\b a}$ be the $d$-dimensional one-row Hermite normal form simplex associated with $\b a=(N-1,\ldots,N-1,N)$, where $N>1$. 
Then the following hold:
\begin{enumerate}
    \item if $N=kd+1$ for some $k\in\N$, then $\polytope{S}_{\b a}$ admits a flag, regular, and unimodular triangulation;
    \item if $N=kd$ for some $k\in\N$, then $\polytope{S}_{\b a}$ admits a regular and unimodular triangulation, which is also flag when $d=2$; 
    \item if $N$ is not of the form $kd$ or $kd+1$ for any $k\in\N$, then $\polytope{S}_{\b a}$ does not satisfy the integer decomposition property, and hence it does not admit a unimodular triangulation.
\end{enumerate}
\end{maintheorem}

The positive cases are proved constructively. 
We first describe the lattice points of $\polytope{S}_{\b a}$ and then use those lying on a distinguished line to construct a collinear cone triangulation. 
This keeps the triangulation explicit and relies only on combinatorial and geometric arguments. 
The same triangulations also allow us to compute the $h^\ast$-polynomial by means of shellings, leading to the following formulas.
\begin{maintheorem}[\Cref{thm:h_vector_OneR},  \Cref{thm:Local_one_row_dk}] \label{thm:B}
Let $\polytope{S}_{\b a}$ be the one-row Hermite normal form simplex associated with 
$\b a=(N-1,\ldots,N-1,N)$. Then
$$
h^\ast_{\polytope{S}_{\b a}}(t)=
\begin{cases}
1+k\displaystyle\sum_{j=1}^{d-1}t^j+(k-1)t^d, & \text{if } N=dk,\\[0.3em]
1+k\displaystyle\sum_{j=1}^{d}t^j, & \text{if } N=dk+1.
\end{cases}
$$
Moreover, the corresponding local $h^\ast$-polynomials are
$$
\ell^\ast_{\polytope{S}_{\b a}}(t)=
\begin{cases}
(k-1)\displaystyle\sum_{i=1}^{d}t^i, & \text{if } N=dk,\\[0.3em]
k\displaystyle\sum_{i=1}^{d}t^i, & \text{if } N=dk+1.
\end{cases}
$$
\end{maintheorem}

These formulas also allow us to study the Ehrhart polynomial of this family of simplices. 
Ehrhart positivity and the unimodality of the coefficient sequence of the Ehrhart polynomial are classical problems in Ehrhart theory. 
Liu and Solus provide a comprehensive overview of results on these questions for several families of lattice polytopes in \cite{Solus_Fu_Unimodal}.
For the family considered here, we prove Ehrhart positivity in the cases $N=dk$ and $N=dk+1$. 
However, Ehrhart positivity does not force Ehrhart unimodality in this family. 
More precisely, \Cref{thm:explicit_non_unimodality_bound} gives explicit
dimension-dependent bounds on $k$ that guarantee non-unimodality in each of the
two cases. 

The final part of the paper shows that our constructions can be extended beyond the one-row setting.
This extension is motivated by the one-row simplex associated with $\b a=(1,\ldots,1,N)$ (studied in \cite{Bajo2025LocalF}), which does not satisfy the integer decomposition property for $d>2$ and $N>1$. 
In contrast, after adding a second row of the form $\b b=(M-1,\ldots,M-1,M,0)$, the corresponding two-row simplex admits a unimodular triangulation in two natural cases, which in particular implies that it satisfies the integer decomposition property.

\begin{maintheorem}[\Cref{thm_UT_2}]\label{thm:C}
Let $\polytope{S}_{\b a,\b b}$ be the two-row Hermite normal form simplex associated with
$\b a=(1,\ldots,1,N)$ and $\b b=(M-1,\ldots,M-1,M,0)$, where $N>1$.
If $M=d-1$ or $M=d$, then $\polytope{S}_{\b a,\b b}$ admits a regular and unimodular triangulation.
When $d=3$ and $M=d-1$, the triangulation can additionally be chosen flag.
\end{maintheorem}

For this two-row family, we also compute the $h^\ast$-polynomials and the local
$h^\ast$-polynomial, completely classify when these simplices have the integer decomposition property for $M\leq d$, and
prove Ehrhart positivity.
Specifically, \Cref{coro:idp_classification_two_rows} shows that, for $M\leq d$, both, 
the integer decomposition property and the existence of a unimodular triangulation, are equivalent to
$M\in\{d-1,d\}$.  Moreover, \Cref{thm:local_h_two_rows} gives
$$
\ell^\ast_{\polytope{S}_{\b a,\b b}}(t)=
\begin{cases}
0, & \text{if } M=d-1,\\[0.2em]
(N-1)\displaystyle\sum_{i=1}^d t^i, &\text{if }M=d.
\end{cases}
$$
If $M=d$, the Ehrhart polynomial has the same form as in the one-row case
$N=dk$, with $k$ replaced by $N$; hence it is not unimodal whenever
$N\geq K_D(d)$ in the notation of
\Cref{thm:explicit_non_unimodality_bound}.

The paper is organized as follows. 
In \Cref{sect:prel}, we recall the necessary background on Hermite normal form simplices, Ehrhart theory, unimodular triangulations, and collinear cone triangulations. 
 \Cref{sect:IDP} focuses on the  integer decomposition property and proves \Cref{thm:A}. 
In \Cref{sect:1row}, we compute the $h^\ast$-polynomial, the local $h^\ast$-polynomial, and the Ehrhart polynomial of the considered  one-row family and, in particular, we prove \Cref{thm:B}.
 \Cref{sect:2row} contains all results concerning two-row Hermite normal form
simplices, including the proofs of \Cref{thm:C} and \Cref{thm:local_h_two_rows}. 

\section{Preliminaries}\label{sect:prel}

In this section, we provide the necessary background on  Hermite normal form simplices,  Ehrhart theory and unimodular triangulations. 
We will assume basic knowledge on polytopes and refer to \cite{Bajo2025LocalF, Unimodular_triangulation_positive_result, beck2015computing, Ziegler} for further details.

\subsection{Hermite normal form simplices}\label{sub:Hermite_normal_form}

A  polytope $\polytope{P}\subset \R^d$ is a called a \defn{lattice polytope} if all its vertices belong to $\Z^d$. 
In particular, $\polytope{P}$ is a $d$-dimensional \defn{lattice simplex} if it is the convex hull of $d+1$ affinely independent lattice points.
Two lattice polytopes $\pol,\polytope{Q}\subseteq\R^d$ are \defn{unimodularly equivalent} if $\polytope{Q}=\varphi(\pol)$ for some \defn{affine unimodular transformation} $\varphi$ (\ie  $\mathdefn{\varphi(\b x)}=U\b x+\b b$, where $U$ is an integer matrix satisfying $|\det(U)|=1$ and $\b b\in\Z^d$).
This notion allows us to represent every lattice simplex, up to unimodular equivalence, in a particularly useful normal form (see e.g., \cite[Section~4.1]{MR874114}). 

\begin{theorem}
Every lattice simplex $\polytope{S}$ of dimension $d$ in $\R^d$ is unimodularly equivalent to a simplex $\polytope{S}_H$ defined as the convex hull of the rows of a $(d + 1) \times d$ integer matrix $H=(a_{ij})_{\substack{0\leq i\leq d\\1\leq j\leq d}}$ of the following form:
\begin{itemize}
    \item $a_{0,j} = 0 $ for $ j = 1, \ldots, d$
    \item $a_{i,i}\in\N$ for $ i = 1, \ldots, d$
    \item $0\leq a_{i,j} < a_{i,i}$ when $j<i$ for $ i = 1, \ldots, d$
    \item $a_{i,j} = 0$ for $j>i$ for $ i = 1, \ldots, d$
\end{itemize}
Moreover, $H$ is uniquely determined and called the \defn{Hermite normal form} of $\polytope{S}$. 
\end{theorem}

We refer $\polytope{S}_H$ as the \defn{Hermite normal form simplex} associated with the matrix $H$. 
We denote by $A$ the matrix obtained by appending a leading column of ones to $H$, which corresponds to embedding $\polytope{S}_H$ at height one in a space of one  dimension higher.
We recall the following definition and notation from \cite{Bajo2025LocalF}.

\begin{definition}\label{def:1-row}
We say that a simplex $\polytope{S}$ is in \defn{one-row Hermite normal form} 
if its Hermite normal form matrix $H$ has the property that $a_{i,i} = 1$ for all 
$1 \leq i \leq d - 1$. In other words, all but the last diagonal entry of the extended matrix $A$  are equal to 1, 
and only the bottom-right entry $a_{d,d} = N$ is allowed to be larger.
In this case, we refer to such a simplex using only the last row 
$\b a = (a_1, \ldots, a_{d-1}, N)$, where $a_i \coloneqq a_{d,i}$ and $a_d \coloneqq N$. We use $\polytope{S}_{\b a}$ to denote the Hermite normal form simplex associated to the row $\b a$. \end{definition}

\begin{example}\label{ex:Matrix_H_A}
For the sequence $\b a = (2,3,4,4,5)$, the Hermite normal form $H$ and its extended form $A$ are as follows:
\begin{center}
\begin{tabular}{ c c c c c }
 $H = \begin{pmatrix}
0 & 0 & 0 & 0 & 0\\
1 & 0 & 0 & 0 & 0\\
0 & 1 & 0 & 0 & 0\\
0 & 0 & 1 & 0 & 0\\
0 & 0 & 0 & 1 & 0\\
2 & 3 & 4 & 4 & 5
\end{pmatrix}$ &  &  & &  $A = \begin{pmatrix}
1 & 0 & 0 & 0 & 0 & 0\\
1 & 1 & 0 & 0 & 0 & 0\\
1 & 0 & 1 & 0 & 0 & 0\\
1 & 0 & 0 & 1 & 0 & 0\\
1 & 0 & 0 & 0 & 1 & 0\\
1 & 2 & 3 & 4 & 4 & 5
\end{pmatrix}$ 
\end{tabular}   
\end{center}
\end{example}

\subsection{Ehrhart theory and unimodular triangulations}

The \defn{Ehrhart polynomial} of a lattice $d$-polytope $\polytope{P}\subset \R^d$ is defined by $\mathdefn{\mathcal{L}_{\polytope{P}}(m)} \coloneqq \lvert\, m\polytope{P} \cap \mathbb{Z}^d \,\rvert$, where $\mathdefn{m\polytope{P}} \coloneqq \{\, m\b p ~:~ \b p \in \polytope{P} \,\}$ denotes the $\mathdefn{m\textsuperscript{th}}$ \defn{dilation} of $\polytope{P}$ for $m\in\Z_{\geq 0}$.
By Ehrhart's theorem (see \cite{Ehrhart_poly}), $\mathcal{L}_{\polytope{P}}(m)$ agrees with a polynomial in $m$ of degree $d$.
The corresponding \defn{Ehrhart series} is the rational function
$$\mathdefn{\operatorname{Ehr}_{\polytope{P}}(t)} \coloneqq \sum_{m\geq 0} \mathcal{L}_{\polytope{P}}(m) t^m =\frac{h^{\ast}_0 + h^{\ast}_1 t + \cdots + h^{\ast}_d t^d}{(1-t)^{d+1}}, $$
where $\mathdefn{h^{\ast}_{\polytope{P}}(t)}=h^{\ast}_0 + h^{\ast}_1 t + \cdots + h^{\ast}_d t^d$ and   $h^\ast(P)\coloneqq (h^{\ast}_0,\ldots,h^{\ast}_d)$ are called the \defn{$h^{\ast}$-polynomial}  and   the \defn{$h^{\ast}$-vector} of $\polytope{P}$, respectively. 
Stanley's non-negativity theorem (see \cite{stanley1980decompositions}) states that $h^{\ast}_i\in \N_{\geq 0}$ for all $0 \leq i \leq d$ which motivates  studying combinatorial properties of $h^\ast(P)$. 
In particular, the \defn{unimodality}\footnote{A sequence $(l_0,\ldots,l_d)$ is \defn{unimodal} if there exists an index $0 \leq i \leq d$ such that $l_0 \leq \cdots \leq l_i \geq \cdots \geq l_d$.} of the coefficient sequences of $h^{\ast}_{\polytope{P}}(t)$ and $\mathcal{L}_{\polytope{P}}(t)$ has received considerable attention (see  \cite{Bajo2025LocalF, Hibi_Hermite, Solus_Fu_Unimodal}).

The $h^{\ast}$-vector can also be computed via triangulations.
A \defn{(lattice) triangulation} of a $d$-dimensional lattice polytope $\polytope{P}\subset \R^d$ is a collection $\mathcal{T}$ of $d$-dimensional lattice simplices such that 
\begin{enumerate}
    \item the intersection of any two simplices in $\mathcal{T}$ is a common face of both, and
    \item the union of all simplices in $\mathcal{T}$ equals $\polytope{P}$.
\end{enumerate}
A triangulation $\mathcal{T}$ of a polytope $\polytope{P}$ is \defn{regular} if there exists a function $\omega : \mathrm{Vert}(\mathcal{T}) \to \R$ such that $\mathcal{T}$ is obtained as the projection of the lower faces of the convex hull of $\{(\b v,\omega(\b v)) : \b v\in \mathrm{Vert}(\mathcal{T})\}\subset\R^{d+1}$, where $\mathrm{Vert}(\mathcal{T})$ denotes the set of vertices of any simplex in $\mathcal{T}$. The triangulation $\mathcal{T}$ is \defn{flag} if all its minimal non-faces have size $2$, and it is called \defn{unimodular} if every simplex in $\mathcal{T}$ has normalized volume equal to $1$.

We will repeatedly use the following elementary volume criterion.

\begin{lemma}\label{lem:volume_count_unimodular}
Let $\mathcal T$ be a lattice triangulation of a $d$-dimensional lattice
polytope $\polytope{P}$.  If the number of $d$-simplices of $\mathcal T$ is equal to the normalized volume 
$\nVol(\polytope{P})$ of $\polytope{P}$, then $\mathcal T$ is unimodular.
\end{lemma}

\begin{proof}
The normalized volume of each maximal lattice simplex is a positive integer,
and the sum of these volumes is $\nVol(\polytope{P})$.  If their number already
equals this sum, every summand must be equal to $1$.
\end{proof}

We say that $\polytope{P}$ admits the \defn{integer decomposition property }(IDP, for short) if for every positive integer $m$ and every $\b p \in m\polytope{P} \cap \Z^d$, there exist lattice points $\b p_1, \ldots, \b p_m \in \polytope{P} \cap \Z^d$ such that $ \b p = \b p_1 + \cdots + \b p_m.$

The \defn{Ehrhart ring} of $\polytope{P}$ is the graded $\mathbb K$-algebra (where $\mathbb K$ is a field)
$$
\mathdefn{\mathbb K[\polytope{P}]}
\coloneqq
\bigoplus_{m\geq 0}\operatorname{span}_{\mathbb K}
\{\b x^{\b a}z^m:\b a\in m\polytope{P}\cap\Z^d\}.
$$

It is easy to see (and well known) that ist Hilbert series equals the Ehrhart series of $\polytope{P}$.
On the other hand, let $\mathdefn{A_{\polytope{P}}}$ denote the graded $\mathbb K$-algebra generated by the monomials $\b x^{\b a}z$ with $\b a\in\polytope{P}\cap\Z^d$. Then, obviously, $\polytope{P}$ is IDP if and only if   $A_{\polytope{P}}=\mathbb K[\polytope{P}]$ (see e.g., \cite[Section 2.C]{Bruns}), which can be tested by comparing the Hilbert series. 

Furthermore, if a lattice polytope $\polytope{P}$ admits a unimodular triangulation, then $\polytope{P}$ has the IDP (see \cite[Sec.~1.2.5]{Unimodular_triangulation_positive_result}).
Therefore, the study of unimodular triangulations and the IDP provides a natural bridge between discrete geometry and combinatorial commutative algebra.

Moreover, we now recall that the $h^{\ast}$-vector admits a combinatorial interpretation in terms of a \emph{shelling order} of a unimodular triangulation of a lattice polytope $\polytope{P}$.

\begin{definition}\label{def:shelling_order}
Let $\mathcal{T}$ be a triangulation of a lattice polytope $\polytope{P}$.
An ordering $\mathcal{T}_1, \ldots, \mathcal{T}_s$ of the simplices in $\mathcal{T}$ is called a \defn{shelling order} if for every $1 < r \leq s$, the intersection
$$
\bigcup_{i = 1}^{r-1} (\mathcal{T}_i \cap \mathcal{T}_r)
$$
is a union of facets of $\mathcal{S}_r$.
\end{definition}

The relevance of shelling orders for our purposes is given by the following result.

\begin{theorem}[\cite{stanley1980decompositions}]\label{thm:Shelling_order}
Let $\polytope{P}$ be a $d$-dimensional lattice polytope and let $\mathcal{T} = \{\mathcal{T}_1, \ldots, \mathcal{T}_s\}$ be a unimodular triangulation of $\polytope{P}$.
If $\mathcal{T}_1, \ldots, \mathcal{T}_s$ is a shelling order, then
$$
h^{\ast}_{\polytope{P}}(t) = \sum_{i = 1}^{s} t^{\omega_i},
$$
where $\omega_i = \# \left\{ k < i ~:~ \mathcal{T}_k \cap \mathcal{T}_i \text{ is a facet of } \mathcal{T}_i \right\}$.
\end{theorem}

\subsection{The local $h^\ast$-polynomial}

Let $\polytope{S}$ a simplex with associated Hermite normal form $H$ and extended matrix $A$ (see \Cref{sub:Hermite_normal_form}). Let $\mathdefn{\Lambda}=\mathdefn{\Lambda(\polytope{S}_{\b a})}\coloneqq \Z^{d+1}A^{-1}$ and $\mathdefn{\Gamma}=\mathdefn{\Gamma(\polytope{S})}\coloneqq \Lambda/\Z^{d+1}$. 
We refer to $\Gamma$ as the \defn{parallelepiped group} associated with $\polytope{S}$.
For $\b x=(x_0,\dots,x_d)\in\Lambda$, let $\mathdefn{\operatorname{frac}(x_i)}\coloneqq x_i-\lfloor x_i\rfloor$ for each $0\leq i\leq d$, and define
$$
\mathdefn{\operatorname{frac}(\b x)}\coloneqq (\operatorname{frac}(x_0),\dots,\operatorname{frac}(x_d))
\qquad\text{and}\qquad
\mathdefn{\operatorname{age}(\b x)}\coloneqq \sum_{i=0}^d \operatorname{frac}(x_i).
$$
Since $\operatorname{age}(\b x)$ depends only on the class of $\b x$ modulo $\Z^{d+1}$, it induces a well-defined function on $\Gamma$.
Moreover, each class in $\Gamma$ admits a unique representative in $[0,1)^{d+1}$, obtained by taking coordinatewise fractional parts. Thus, we identify $\Gamma$ with this set of representatives.

The \defn{local $h^{\ast}$-polynomial} (also called \defn{box polynomial}) of $\polytope{S}$ is given by
\begin{equation}\label{eq:local_h_S}
\mathdefn{\ell^{\ast}_{\polytope{S}_{\b a}}(t)}=\sum_{\b x\in \Gamma\cap(0,1)^{d+1}} t^{\operatorname{age}(\b x)}.
\end{equation}
If $\polytope{S}$ is a simplex in one-row Hermite normal form, \ie $\polytope{S}=\polytope{S_{\ ba}}$ for some Let $\b a\in \mathbb{N}^d$, then  all elements of $\Gamma$ can be written as $\mathdefn{\b x_m}=\operatorname{frac}(m\overline{\b a})$ for $m=0,1,\dots,N-1$, where $\overline{\b a}$ denotes the last row of $A^{-1}$ (see Section 2.2. in \cite{Bajo2025LocalF} for details).

\subsection{Collinear cone triangulations}

\begin{definition}\cite[Definition 2.6]{s_lecture_reference}
Let $\polytope{P}$ be a lattice polytope of dimension $d$ and let $\b x \in \Z^d$ such that $\b x \notin \polytope{P}$.
Define the \defn{one-point extension} of $\polytope{P}$ by $\b x$ as the polytope $\conv(\polytope{P} \cup \{\b x\})$. If $\b x \notin \aff(\polytope{P})$, then we call $\mathdefn{ \Pyr(\polytope{P}, \b x)} \coloneqq\conv(\polytope{P} \cup \{\b x\})$  the \defn{pyramid over $\polytope{P}$ with apex $\b x$}. 
Moreover, we say that a face $\polytope{F}$ of $\polytope{P}$ is \defn{visible from $\b x$} if $\conv(\polytope{F} \cup  \{\b x\}) \cap \polytope{P} = \polytope{F}$.
\end{definition}
The next lemma allows us to build triangulations for successive one-point extensions. 

\begin{lemma}\label{lem:Iterated_one_point_extension}
Let $\Delta_0 \subset \R^d$ be a lattice $d$-polytope and let $\mathcal{T}_0$ be a triangulation of $\Delta_0$.
Let $\b x_1,\dots,\b x_k \in \Z^d$ be lattice points such that $\b x_i \notin \Delta_{i-1}$ for each $1 \leq i \leq k$, where
$\Delta_i \coloneqq \conv\big(\Delta_{i-1} \cup \{\b x_i\}\big)$.
For $1\leq i\leq k$, let $\mathcal{F}_{\b x_i}$ be the set of facets of $\Delta_{i-1}$ visible from $\b x_i$ and set
$$
\mathcal{T}_i \coloneqq \mathcal{T}_{i-1} \cup \big\{\conv(S \cup \{\b x_i\}) : S \in \mathcal{T}_{i-1}\vert_{\mathcal{F}_{\b x_i}}\big\},
$$
where $\mathcal{T}_{i-1}\vert_{\mathcal{F}_{\b x_i}}$
denotes the triangulation induced by $\mathcal{T}_{i-1}$ on
$\bigcup_{\polytope{F}\in\mathcal{F}_{\b x_i}} \polytope{F}$.

Then $\mathcal{T}_i$ is a triangulation of $\Delta_i$ for every $1 \leq i \leq k$.
\end{lemma}

\begin{proof}
We argue by induction on $i$. The base case $i=0$ holds by assumption.
Assume that $\mathcal{T}_{i-1}$ is a triangulation of $\Delta_{i-1}$ for some $1 \leq i \leq k$.
Since $\b x_i \notin \Delta_{i-1}$, the polytope $\Delta_i=\conv(\Delta_{i-1}\cup\{\b x_i\})$ is the one-point extension of $\Delta_{i-1}$ by $\b x_i$. 
Applying \cite[Theorem~2.7]{s_lecture_reference} to $\Delta_{i-1}$, $\mathcal{T}_{i-1}$, and $\b x_i$, we obtain that the complex $\mathcal{T}_i$ defined above is a triangulation of $\Delta_i$.
This completes the induction and the proof.
\end{proof}

\begin{remark}\label{remark:flag_regular}
By \cite[Theorem~2.7]{s_lecture_reference}, if $\mathcal{T}_0$ is regular (respectively flag), then $\mathcal{T}_k$ is also regular (respectively flag).
\end{remark}

\begin{definition}\label{def:colinear_cone_triangulation}
Let $\polytope{P}\subset\R^d$ be a $d$-dimensional lattice simplex with vertices $\mathcal{V}=\{\b v_0,\dots,\b v_d\}$ and let  $\b v\in\mathcal{V}$ be fixed. 
Define $\polytope{P}_0\coloneqq\conv(\mathcal{V}\setminus\{\b v\})$. 
Suppose that $\polytope{P}\cap\Z^d=(\polytope{P}_0\cap\Z^d)\sqcup\{\b v\}\sqcup \mathcal{I}$, where $\mathcal{I}\subset\polytope{P}$ is a finite set of lattice points that, together with $\b v$, lie on a common line $L$ and are different from the vertices of $\polytope P$.
Given a triangulation $\mathcal{T}$ of $\polytope{P}_0$, the \defn{collinear cone triangulation} of $\polytope{P}$ is the triangulation obtained from $(\polytope{P}_0,\mathcal{T})$ by iteratively applying \Cref{lem:Iterated_one_point_extension}, where one starts with the point of $\mathcal{I}$ closest to $\polytope{P}_0$, goes along $L$ and ends with the vertex $\b v$.
\end{definition}

An example of a collinear cone triangulation is shown in \Cref{fig:Collinear_cone_example}, where $\polytope{P} = \conv(\b v_0, \b v_1,\b v_2)$ and $\polytope{P}_0 = \conv(\b v_0, \b v_1)$.

\begin{figure}[ht]
\centering
\tdplotsetmaincoords{70}{70} 
\begin{tikzpicture}[tdplot_main_coords, scale=1.8]

\definecolor{myblue}{rgb}{0.6,0.8,1}
\definecolor{mygreen}{rgb}{0.65,0.9,0.65}
\definecolor{myred}{rgb}{1,0.7,0.7}
\definecolor{myyellow}{rgb}{1,1,0.6}
\definecolor{mypurple}{rgb}{0.85,0.75,1}

\begin{scope}[xshift=0cm]

\coordinate (v1) at (1, 0, 0);
\coordinate (v2) at (1, 2, 0);
\coordinate (v3) at (1, 1, 5);
\coordinate (p0)  at (1, 1, 0);
\coordinate (p1)  at (1, 1, 1);
\coordinate (p2)  at (1, 1, 2);
\coordinate (p3)  at (1, 1, 3);
\coordinate (p4)  at (1, 1, 4);


\filldraw[fill=mypurple, opacity=0.6] (v2) -- (p4) -- (v3) -- cycle;

\filldraw[fill=myyellow, opacity=0.6] (v2) -- (p3) -- (p4) -- cycle;

\filldraw[fill=myred, opacity=0.6] (v2) -- (p2) -- (p3) -- cycle;

\filldraw[fill=mygreen, opacity=0.6] (v2) -- (p1) -- (p2) -- cycle;

\filldraw[fill=myblue, opacity=0.6] (v2) -- (p1) -- (p0) -- cycle;

\filldraw[fill=myblue, opacity=0.6] (v1) -- (p1) -- (p0) -- cycle;

\filldraw[fill=mygreen, opacity=0.6] (v1) -- (p1) -- (p2) -- cycle;

\filldraw[fill=myred, opacity=0.6] (v1) -- (p2) -- (p3) -- cycle;

\filldraw[fill=myyellow, opacity=0.6] (v1) -- (p3) -- (p4) -- cycle;

\filldraw[fill=mypurple, opacity=0.6] (v1) -- (p4) -- (v3) -- cycle;

\foreach \pt/\name in { v1/$\b v_0$, v2/$\b v_1$, v3/$\b v_2$, p0/$\b p_0$, p1/$\b p_1$, p2/$\b p_2$, p3/$\b p_3$, p4/$\b p_4$} {
  \filldraw[black] (\pt) circle (0.4pt) node[anchor=north east] {\name};
}
\end{scope}
\end{tikzpicture} 
\caption{The collinear cone triangulation of $ \conv(\b v_0, \b v_1, \b v_2)$ along the lattice points $\b p_0, \ldots, \b p_4, \b v_2$.}
\label{fig:Collinear_cone_example}
\end{figure}
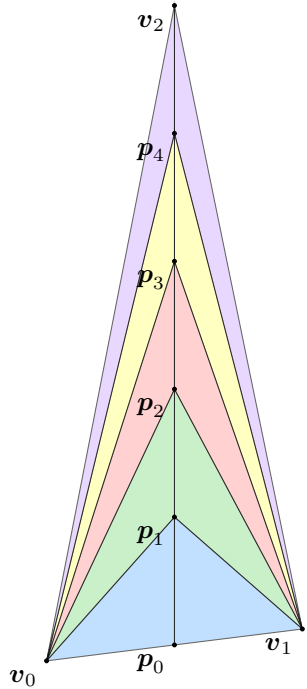

\begin{lemma}\label{lem:shelling_collinear_cone}
Let $\polytope{P}\subseteq\R^d$ be a $d$-dimensional lattice simplex with vertices $\mathcal{V}=\{\b v_0,\dots,\b v_d\}$, and let $\b v\in\mathcal{V}$ be fixed
Let $\polytope{P}_0\coloneqq \conv(\mathcal{V}\setminus\{\b v\})$. Assume further that $\polytope{P}_0\cap\Z^d$ consists of its vertices and possibly one interior lattice point. 
If $\polytope{P}$ admits a collinear cone triangulation $\mathcal{T}$ with distinguished vertex $\b v$ and base facet $\polytope{P}_0$, then $\mathcal{T}$ admits a shelling order.
\end{lemma}

\begin{proof}
After relabeling, we may assume that $\b v=\b v_0$, so that 
$\polytope{P}_0=\conv(\b v_1,\dots,\b v_d)$.
By assumption, there exists a line $L$ that contains all lattice  points of $\polytope{P}\setminus \polytope{P}_0$. We write $\b p_1,\dots,\b p_m$ for these lattice points, where they are numbered according to the order in which they are added in the collinear triangulation. In particular, $\b p_1$ is the  point closest to $\polytope{P}_0$ and $\b p_m=\b v_0$. 
For $1\leq j\leq d$, let $F_j \coloneqq \conv\big(\{\b v_1,\dots,\b v_d\}\setminus\{\b v_j\}\big)$.
Suppose that $\polytope{P}_0$ has a unique interior lattice point $\b p_0$. 
Since the induced triangulation of $\polytope{P}_0$ uses all its lattice points, its maximal simplices are precisely $\conv(p_0,F_1),\ldots,\conv(\b p_0,F_d)$. 
For $1\leq r\leq m$ and $1\leq j\leq d$, define $\sigma_{r,j}:=\conv(\b p_{r-1},\b p_r,F_j)$. 
By construction, the maximal simplices of the collinear cone triangulation $\mathcal T$ are exactly the simplices $\sigma_{r,j}$.
Consider the order
$$
\sigma_{m,1},\dots,\sigma_{m,d},\sigma_{m-1,1},\dots,\sigma_{m-1,d},\dots,\sigma_{1,1},\dots,\sigma_{1,d}.
$$
We show that this is a shelling order. 
For the first layer, if $1\leq j\leq d$, then the only facets of $\sigma_{m,j}$ contained in the union of the preceding simplices are
$$
\sigma_{m,j}\cap \sigma_{m,\ell}=\conv(\b p_{m-1},\b p_m,F_j\cap F_\ell)
\qquad \text{for }1\leq \ell<j.
$$
Now let $1\leq r\leq m-1$.
For each $1\leq j\leq d$, the simplex $\sigma_{r,j}$ has one facet coming from the previous layer, namely $\sigma_{r,j}\cap \sigma_{r+1,j}=\conv(\b p_r,F_j)$, and exactly $j-1$ additional facets coming from the simplices $\sigma_{r,1},\dots,\sigma_{r,j-1}$ in the same layer.
All remaining intersections with preceding simplices have dimension at most $d-2$.
Hence, the above ordering is a shelling order of $\mathcal{T}$.

Now assume that $\polytope{P}_0$ has no interior lattice point. 
Let $\tau\coloneqq \conv(\polytope{P}_0,\b p_1)$, and for $1\leq r\leq m-1$ and $1\leq j\leq d$, define $\sigma_{r,j}\coloneqq \conv(\b p_r,\b p_{r+1},F_j)$.
By construction, the maximal simplices of $\mathcal{T}$ are $\tau$ together with the simplices $\sigma_{r,j}$ for $1\leq r\leq m-1$ and $1\leq j\leq d$.
Consider the order
$$
\tau,\sigma_{1,1},\dots,\sigma_{1,d},\sigma_{2,1}\dots \sigma_{2,d}, \dots , \sigma_{m-1,1},\dots,\sigma_{m-1,d}.
$$
This order starts with the simplex $\tau$.
For $1\leq j\leq d$, the simplex $\sigma_{1,j}$ has one facet contained in $\tau$, namely
$$
\sigma_{1,j}\cap \tau=\conv(\b p_1,F_j),
$$
and exactly $j-1$ additional facets intersecting the simplices $\sigma_{1,1},\dots,\sigma_{1,j-1}$.
Now let $2\leq r\leq m-1$.
For each $1\leq j\leq d$, the simplex $\sigma_{r,j}$ has one facet coming from the preceding layer, namely
$$
\sigma_{r,j}\cap \sigma_{r-1,j}=\conv(\b p_r,F_j),
$$
and exactly $j-1$ additional facets coming from the simplices $\sigma_{r,1},\dots,\sigma_{r,j-1}$ in the same layer.
All remaining intersections with preceding simplices have dimension at most $d-2$.
Hence, the given order is indeed a shelling order of $\mathcal{T}$.
\end{proof}

\begin{theorem}\label{thm:hstar_collinear_cone}
Let $\polytope{P}\subset \mathbb{R}^d$ be a $d$-dimensional lattice simplex admitting a collinear cone triangulation $\mathcal{T}$ with distinguished vertex $\b v$, base facet $\polytope{P}_0$ and line $L$. 
Assume that $\polytope{P}_0$ has no boundary lattice points other than its vertices and at most one interior lattice point. Let $m=\#(L\cap \polytope{P^\circ}\cap \mathbb{Z}^d)$. 
If $\mathcal{T}$ is unimodular, then
$$
h^\ast_{\polytope{P}}(t)=h^\ast_{\polytope{P}_0}(t)+(m-1)(t+t^2+\cdots+t^d).
$$

\end{theorem}

\begin{proof}
Since $\mathcal{T}$ is unimodular, by \Cref{thm:Shelling_order}, it is enough to sum $t^{\omega(\sigma)}$ over the maximal simplices $\sigma$ of $\mathcal{T}$, where $\omega(\sigma)$ is computed with respect to the shelling order constructed in the proof of \Cref{lem:shelling_collinear_cone}. Let $\b v_1,\ldots,\b v_d$ be the vertices of $\polytope{P}_0$ and for  $1\leq j\leq d$, let $F_j$ be the facet of $\polytope{P}_0$ not containing $v_j$. 
Assume first that $\polytope{P}_0$ has a unique interior lattice point $\b p_0$. 
For $1\leq r\leq m$ and $1\leq j\leq d$, as in the proof of \Cref{lem:shelling_collinear_cone}, set $\sigma_{r,j}\coloneqq \conv(\b p_{r-1},\b p_r,F_j)$. 
The proof of \Cref{lem:shelling_collinear_cone} shows that  $\omega(\sigma_{m,j})=j-1$ for $1\leq j\leq d$, and $\omega(\sigma_{r,j})=j$ for $1\leq r\leq m-1$ and $1\leq j\leq d$. 
Hence 
$$ h^\ast_{\polytope{P}}(t)=\sum_{j=1}^d t^{j-1}+(m-1)\sum_{j=1}^d t^j.$$
Since the induced triangulation of $\polytope{P}_0$ is unimodular, it must use $\b p_0$; hence its maximal simplices are precisely $\conv(\b p_0,F_1),\ldots,\conv(\b p_0,F_d)$. 
With this shelling order, $h^\ast_{\polytope{P}_0}(t)=1+t+\cdots+t^{d-1}$.
Thus
$
h^\ast_{\polytope{P}}(t)=h^\ast_{\polytope{P}_0}(t)+(m-1)(t+t^2+\cdots+t^d).
$

Now assume that $\polytope{P}_0$ has no interior lattice point. 
Let $\tau\coloneqq \conv(\polytope{P}_0,\b p_1)$, and for $1\leq r\leq m-1$ and $1\leq j\leq d$, define $\sigma_{r,j}\coloneqq \conv(\b p_r,\b p_{r+1},F_j)$. 
By the construction in the proof of \Cref{lem:shelling_collinear_cone}, one has $\omega(\tau)=0$ and $\omega(\sigma_{r,j})=j$ for $1\leq r\leq m-1$ and $1\leq j\leq d$. 
Hence $h^\ast_{\polytope{P}}(t) = 1 + (m-1)(t + \dots + t^d)$.
In this case, the induced triangulation is trivial, and therefore $h^\ast_{\polytope{P}_0}(t)=1$. 
Therefore $ h^\ast_{\polytope{P}}(t)=h^\ast_{\polytope{P}_0}(t)+(m-1)(t+t^2+\cdots+t^d).$
This proves the theorem.
\end{proof}

\section{Unimodular triangulations and the integer decomposition property }
\label{sect:IDP}

In this section, we characterize when the one-row Hermite normal form simplex $\polytope{S}_{\b a}$ associated with $\b a = (N - 1, \dots, N - 1, N) \in \N^d$ admits a unimodular triangulation. In particular, we provide the proof of \Cref{thm:A}. 
Throughout this section, we assume that $N>1$, since for $N=1$, the simplex $\polytope{S}_{\b a}$ is unimodular itself and hence trivial.
This simplex was previouly studied in \cite[Section~3.1]{Hibi_Hermite}, where the main objective was to investigate the shifted symmetry property of the $h^{\ast}$-vector. 
Our approach is entirely different and complements the results obtained there.

\subsection{Unimodular triangulation}\label{subsection_UT_1row}
 Let $\b a = (a_1, \dots, a_{d-1}, N)\in \mathbb{N}^d$, where $0\leq a_i < N$ for all $1 \leq i \leq d - 1$. 
Let $\polytope{S}_{\b a}$ be the one-row Hermite normal form simplex associated with $\b a$, \ie 

$$
\polytope{S}_{\b a} = \conv\left\{
\underbrace{(0, \dots, 0)}_{\b v_0},\,
\underbrace{(1, 0, \dots, 0)}_{\b v_1},\,
\dots,\,
\underbrace{(0, \dots, 1, 0)}_{\b v_{d-1}},\,
\underbrace{(a_1, \dots, a_{d-1}, N)}_{\b v_d}
\right\}.
$$

We recall the characterization of \Cref{thm:A} of when $S_{\b a}$ admits a unimodular triangulation 
in the following table:

\begin{center}
\begin{tabular}{|c|c|l|}
\hline
\textbf{Case} & \textbf{Condition on } $N$ & \textbf{IDP and Triangulation} \\
\hline\hline
1 & $N = kd + 1$ & \checkmark IDP, flag regular unimodular triangulation \\
  & (for $k \geq 1$) &  \\
\hline
2 & $N = kd$ & \checkmark IDP, regular unimodular triangulation \\
  & (for $k \geq 1$) &  \\
  & & flag when $d=2$ \\
\hline
3 & $N \notin \{kd, kd+1\}$ & \ding{55} No IDP, no unimodular triangulation \\
  & for any $k \in \mathbb{N}$ &  \\
\hline
\end{tabular}
\end{center}

\noindent
This complete characterization will be proven via \Cref{thm:Unimodular_triangulation}, \Cref{thm:not_IDP_2} and \Cref{thm:Not_IDP_1} below.

\begin{proposition}\label{prop:Interio_points}
Let $\b a = (N - 1, \dots, N - 1,\, N)$, and let $\polytope{S}_{\b a}$ be the one-row Hermite normal form simplex associated with $\b a$.  
If $N = k d$ or $N = k d + 1$ for some positive integer $k$, then $\polytope{S}_{\b a}$ contains exactly its vertices and the points $\b p_i = (i, \dots, i)$ for $1 \leq i \leq k$ as lattice points.
\end{proposition}

\begin{proof}

Let us first assume that $N = k d$. 
Let $\b x=(x_1, \dots, x_d) \in \polytope{S}_{\b a} \cap \Z^d$ be such that $\b x \neq \b v_i$ for all $0 \leq i \leq d$.  
Then $\b x$ can be written as a convex combination: $\b x = \sum_{j = 0}^d \lambda_j \b v_j$ such that $\sum_{j = 0}^d \lambda_j = 1$ and $\lambda_j\neq 1$ for $0\leq i\leq d$.
It follows that
$$ x_j = \lambda_j + \lambda_d(k d - 1) \quad \text{for } 1 \leq j \leq d - 1, \quad \text{and} \quad x_d = k d \lambda_d$$
which implies that  $x_j = x_d + \lambda_j - \lambda_d$ for all $1 \leq j \leq d - 1$.  
Since $x_j$ and $x_d$ are integers, and $\lambda_j - \lambda_d \in (-1, 1)$, we must have $\lambda_j = \lambda_d$ for all $j$.  
Therefore, $\b x = (x_d, \dots, x_d)$ with $x_d \in \N$.
On the other hand, as $\lambda_d = \frac{x_d}{k d}$, we have 

\begin{equation}\label{eq:interior_points_1}
    1 = \lambda_0 + \sum_{j = 1}^d \lambda_j =  \lambda_0 + \frac{x_d}{k}.
\end{equation}
 It follows that $x_d\in \{1,\ldots,k\}$. 
 If $x_d = k$, then $\lambda_0 = 0$ and thus $\b x=\b p_k$. 
If $ x_d =i$ for $1\leq i\leq k - 1$, then $\lambda_0 = 1 - \frac{i}{k} \in (0,1)$, and thus $\b x=\b p_i$. This shows the claim in this case.

The case $N = k d + 1$ follows by almost the same argument.  
The only difference is that, in the final step, one obtains $\lambda_0 + x_d\frac{d}{k d + 1} = 1$, which implies $\lambda_0\in(0,1)$ if and only if $1\leq x_d\leq k$.
\end{proof}

The  following corollary completes the description of the lattice points of $\polytope{S}_{\b a}$ for the remaining values of $N$.

\begin{corollary}\label{coro:No_lattice_point}
Let $\b a = (N - 1, \dots, N - 1,\, N)\in \mathbb{N}^d$, and let $\polytope{S}_{\b a}$ be the one-row Hermite normal form simplex associated with $\b a$. Then:
\begin{enumerate}
    \item[(1)] If $1 < N < d$, then $\polytope{S}_{\b a}$ contains only its vertices as lattice points.
    \item[(2)] If $N = k d + m$ with $2 \leq m \leq d - 1$ for some $k \in \N$, then $\polytope{S}_{\b a}$ contains exactly its vertices and the points $\b p_i = (i, \dots, i)$ for $1 \leq i \leq k$ as lattice points.
\end{enumerate}
\end{corollary}

\begin{proof}

Let $\b x \in \polytope{S}_{\b a} \cap \Z^d$ be a lattice point distinct from the vertices.
By the same argument as in the proof of \Cref{prop:Interio_points}, we know that $\b x = (i, \ldots, i)$ for some $i\in \N$ and $\lambda_j = \lambda_d$ for all $1\leq j \leq d-1$.

To prove (1), suppose that $1 < N < d$. Then
\begin{equation}
\label{eq:helper}
1 = \lambda_0 + \sum_{j = 1}^d \lambda_j = \lambda_0 + \frac{d i}{N}.
\end{equation}
Since $i \geq 1$ and $1 < N < d$, we have that $\frac{d i}{N}> 1$, which implies $\lambda_0 < 0$.
Therefore, the only lattice points contained in $\polytope{S}_{\b a}$ are its vertices.

(2) Suppose that $N = k d + m$ with $2 \leq m \leq d - 1$ and $k \in \N$. 
\eqref{eq:helper} yields
$$
\lambda_0 = 1 - \frac{d i}{k d + m}.
$$
Thus, $\lambda_0 \in (0, 1)$ if and only if
$$
0 < 1 -  \frac{i  d}{k d + m} < 1,
$$
which holds precisely when $1 \leq i \leq k$.
Therefore, the only non-vertex lattice points in $\polytope{S}_{\b a}$ are exactly the points $(i, \dots, i)$ for $1 \leq i \le k$.




\end{proof}

We are now ready to prove \Cref{thm:A}, part (1) and (2), which we first recall:

\begin{theorem}\label{thm:Unimodular_triangulation}
Let $\polytope{S}_{\b a}$ be the $d$-dimensional one-row Hermite normal form simplex associated with $\b a = (N - 1, \dots, N - 1, N)$, where $N = k d$ or $N = k d + 1$ for some $k\geq 1$. 
Then:
\begin{enumerate}
    \item if $N = kd + 1$, then $\polytope{S}_{\b a}$ admits a flag, regular, and unimodular triangulation;
    \item if $N = kd$, then $\polytope{S}_{\b a}$ admits a regular and unimodular triangulation; when $d=2$, this triangulation is also flag.
\end{enumerate}
\end{theorem}
\begin{proof}
By \Cref{prop:Interio_points}, if $\b a=(N-1,\dots,N-1,N)$ with $N=kd$ or $N=kd+1$, then the one-row Hermite normal form simplex $\polytope{S}_{\b a}$ contains no lattice points other than its vertices and the points $\b p_i=(i,\dots,i)$ for $1\leq i\leq k$.
Let $\Delta\coloneqq\conv(\b v_1,\dots,\b v_d)$. 
Then $\polytope{S}_{\b a}=\conv(\Delta,\b v_0)$.

First assume that $N=kd$.
Then $\b p_k\in\relint(\Delta)$.
Let $\mathcal J$ be the star triangulation of $\Delta$ obtained by coning with  $\b p_k$ over the facets of $\Delta$.
Since the hypotheses of \Cref{def:colinear_cone_triangulation} are satisfied with distinguished vertex $\b v_0$, the simplex $\polytope{S}_{\b a}$ admits the collinear cone triangulation $\mathcal T$ induced from $(\Delta,\mathcal J)$.
Since $\mathcal J$ is regular, \Cref{remark:flag_regular} implies that so is $\mathcal T$. 
If $d=2$, then $\Delta$ is a segment and the triangulation $\mathcal J$ is flag; hence
$\mathcal T$ is flag in this case by \Cref{remark:flag_regular}.
Since the triangulation $\mathcal J$ has exactly $d$ maximal simplices of dimension $d-1$, applying the collinear cone triangulation along $\b p_{k-1},\dots,\b p_1,\b v_0$ produces $k$ layers, each consisting of $d$ maximal simplices. Hence, $\mathcal T$ consists of exactly $kd$ maximal simplices.
Since $\nVol(\polytope{S}_{\b a})=kd$ by
\cite[Proposition~2.5]{Bajo2025LocalF}, unimodularity follows from
\Cref{lem:volume_count_unimodular}.

Now, suppose $N=kd+1$. 
Then $\Delta$ contains no lattice points other than its vertices, so we take $\mathcal J=\{\Delta\}$. 
Again, the hypotheses of \Cref{def:colinear_cone_triangulation} are satisfied with distinguished vertex $\b v_0$, and thus $\polytope{S}_{\b a}$ admits the collinear cone triangulation $\mathcal T$ induced from $(\Delta,\mathcal J)$. 
Since $\mathcal J$ is flag and regular, \Cref{remark:flag_regular} implies that $\mathcal T$ is also flag and regular.
The first step produces the simplex $\conv(\b p_k,\b v_1,\dots,\b v_d)$, and each of the remaining $k$ steps contributes $d$ maximal simplices.
Therefore, $\mathcal T$ has exactly $kd+1$ maximal simplices.
Since $\nVol(\polytope{S}_{\b a})=kd+1$, unimodularity again follows from
\Cref{lem:volume_count_unimodular}.
\end{proof}

\begin{remark}
If $N=kd$, then the base triangulation $\mathcal J$ of $\Delta$ is obtained by coning with $\b p_k$ over the facets of $\Delta$. This triangulation is regular, but not flag for $d\geq 3$, since the vertices of $\Delta$ span a clique in the $1$-skeleton of $\mathcal J$, but do not form a face of $\mathcal J$.
\end{remark}
We illustrate \Cref{thm:Unimodular_triangulation} with an example. 
\begin{example}
Let $d = 2$ and suppose  $N = 2\cdot 2 =4$ and $\overline{N} = 2\cdot 2 + 1=5$, where $k = 2$.
Then $\b a = (3, 4)$ and  $ \overline{\b a} = (4, 5)$, and its associated Hermite normal form simplex is given by

$$\polytope{S}_{\b a} = \conv\{ (0, 0),\, (1, 0),\, (3,4)\} \quad \text{ and }\quad \polytope{S}_{\overline{\b a}} = \conv\{ (0, 0),\, (1, 0),\, (4,5)\}.$$

We know that $\nVol(\polytope{S}_{\b a}) = 4$ and $\nVol(\polytope{S}_{\overline{\b a}}) = 5$.
Consider the following simplices:

\begin{multicols}{2}
\raggedcolumns
$$
\begin{aligned}
\mathcal{J}_1 &= \conv\{(0,0),(1,0),(1,1)\}\\[0.6em]
\mathcal{J}_2 &= \conv\{(1,0),(1,1),(2,2)\}\\[0.6em]
\mathcal{J}_3 &= \conv\{(1,1),(2,2),(3,4)\}\\[0.6em]
\mathcal{J}_4 &= \conv\{(0,0),(1,1),(3,4)\}
\end{aligned}
$$

\columnbreak

$$
\begin{aligned}
\mathcal{J}_5 &= \conv\{(0,0),(1,1),(4,5)\}\\[0.6em]
\mathcal{J}_6 &= \conv\{(1,0),(2,2),(4,5)\}\\[0.6em]
\mathcal{J}_7 &= \conv\{(1,1),(2,2),(4,5)\}
\end{aligned}
$$
\end{multicols}

Then $\mathcal{J} = \{\mathcal{J}_1, \mathcal{J}_2, \mathcal{J}_3, \mathcal{J}_4\}$ is a unimodular triangulation of $\polytope{S}_{\b a}$ and  $\overline{\mathcal{J}} = \{\mathcal{J}_1, \mathcal{J}_2, \mathcal{J}_5, \mathcal{J}_6, \mathcal{J}_7\}$ is a unimodular triangulation of $\polytope{S}_{\overline{\b a}}$.
The triangulations are shown in \Cref{fig:Example_triangulation}.

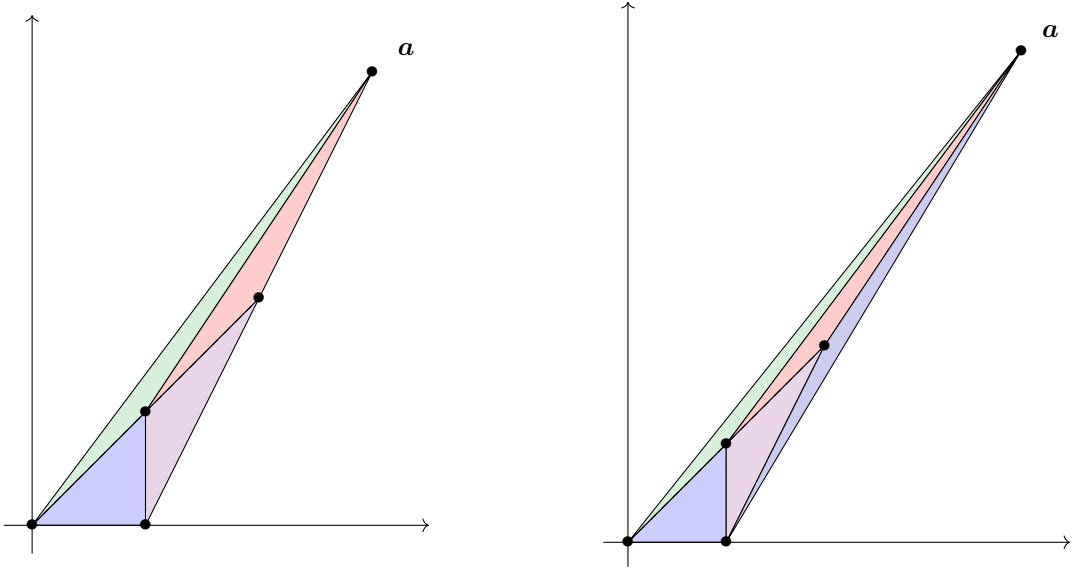
\begin{figure}[h]
\centering
\begin{minipage}{0.45\textwidth}
    \centering
    \usetikzlibrary{patterns,patterns.meta}
\centering \begin{tikzpicture}[scale=1.5]

\draw[fill = blue!20] (0, 0) -- (1,1) -- (1,0) -- cycle;
\draw[fill = green!20] (0, 0) -- (1,1) -- (3,4) -- cycle;
\draw[fill = red!20] (3, 4) -- (1,1) -- (2,2) -- cycle;
\draw[fill = violet!20] (1,0) -- (1,1) -- (2,2) -- cycle;





\draw[black] (3,4) node{$\bullet$};
\draw[black] (0,0) node{$\bullet$};
\draw[black] (1,0) node{$\bullet$};
\draw[black] (1,1) node{$\bullet$};
\draw[black] (2,2) node{$\bullet$};



\draw[black] (3.3,4.2) node{$\b a $};
\draw[dashed, black](0 , 0)--(2 , 2);

\draw[->](-0.25 , 0)--(3.5 , 0);
\draw[->](0,-0.25 )--(0 , 4.5 );

\end{tikzpicture} 
\end{minipage}
\hfill
\begin{minipage}{0.45\textwidth}
    \centering
    \usetikzlibrary{patterns,patterns.meta}
\centering \begin{tikzpicture}[scale=1.3]

\draw[fill = blue!20] (0, 0) -- (1,1) -- (1,0) -- cycle;
\draw[fill = green!20] (0, 0) -- (1,1) -- (4,5) -- cycle;
\draw[fill = red!20] (4,5) -- (1,1) -- (2,2) -- cycle;
\draw[fill = violet!20] (1,0) -- (1,1) -- (2,2) -- cycle;
\draw[fill = darkblue!20] (1,0) -- (2,2) -- (4,5) -- cycle;

\draw[black] (4,5) node{$\bullet$};
\draw[black] (0,0) node{$\bullet$};
\draw[black] (1,0) node{$\bullet$};
\draw[black] (1,1) node{$\bullet$};
\draw[black] (2,2) node{$\bullet$};

\draw[black] (4.3,5.2) node{$\b a $};
\draw[dashed, black](0 , 0)--(2 , 2);

\draw[->](-0.25 , 0)--(4.5 , 0);
\draw[->](0,-0.25 )--(0 , 5.5 );

\end{tikzpicture} 
\end{minipage}
\caption{Left: Unimodular triangulation of $\polytope{S}_{\b a}$. Right: Unimodular triangulation of $\polytope{S}_{\overline{\b a}}$}
\label{fig:Example_triangulation}
\end{figure}

\end{example}

One might hope that the unimodular triangulation constructed above extends to cases where $N\notin\{kd,kd+1\}$, where $d$ is the dimension of $\polytope{S}_{\b a}$.  
However, this is not the case.  
Before analyzing the remaining values of $N$, we need some preliminary results.

\begin{lemma}\label{lem:kS_dilation}
Let $N = k d + m$ with $2 \leq m \leq d - 1$ and $k\in \N_{\geq 1}$ and let $\polytope{S}_{\b a}$ be the one-row Hermite normal form simplex associated with $\b a = (N - 1, \dots, N - 1, N)\in \N^d$. 
Then,  $(kn+1, \dots,  kn+1)\in n  \polytope{S}_{\b a}$ for every positive integer $n \ge \lceil \frac{d}{m} \rceil$.
\end{lemma}

\begin{proof}
Let $n \in \mathbb{N}$ and let $\b p = (k n+1,\dots,k n+1)\in \N^d$ for positive integers $n$ and $k$. 
We write $\b p$ as an affine combination of $n\b v_0,\ldots,n\b v_d$, \ie $\b p= \sum_{i=0}^d \lambda_i  (n \b v_i)$ with $\sum_{i=0}^d \lambda_i = 1$. We need to show that $\lambda_i\geq 0$ for $0\leq i\leq d$. 
Arguing as in \Cref{prop:Interio_points} by comparing coordinates, we obtain
$\lambda_d = \frac{k n+1}{n N}\geq 0$ and $\lambda_j = \lambda_d$ for $1 \leq j \leq d-1$. It remains to show that $\lambda_0\geq 0$. 
Since
$$
1 = \lambda_0 + \sum_{j=1}^d \lambda_j
= \lambda_0 + d\,\lambda_d
= \lambda_0 + \frac{d (k n+1)}{n N}
$$
we have $\lambda_0 \geq 0$ whenever
$$
\frac{d (k n+1)}{n (k d+m)} \leq 1,
$$
which is equivalent to $n \ge \left\lceil \frac{d}{m} \right\rceil$. 
This shows that $\b p \in n\polytope{S}_{\b a}$ for every $n \ge \left\lceil \frac{d}{m} \right\rceil$.

\end{proof}
The next statement is part of \Cref{thm:A} (3).

\begin{theorem}\label{thm:not_IDP_2}
Let $N = k d + m$ with $2 \leq m \leq d - 1$ and $k\in \N_{\geq 1}$ and let $\polytope{S}_{\b a}$ be the one-row Hermite normal form simplex associated with $\b a = (N - 1, \dots, N - 1, N)\in \N^d$. 
Then $\polytope{S}_{\b a}$ does not satisfy the integer decomposition property.
\end{theorem}

\begin{proof}
Let $n = \left\lceil \frac{d}{m} \right\rceil$. 
By \Cref{lem:kS_dilation}, we have $\b x = (n k + 1, \dots, n k + 1)\in n  \polytope{S}_{\b a}$.
Suppose, by contradiction, that $\polytope{S}_{\b a}$ satisfies the IDP. 
Then there exist $\b q_1,\ldots,\b q_n\in\polytope{S}_{\b a}\cap\Z^d$ such that  $\b x=\b q_1+\cdots+\b q_n$.
Since $m\geq 2$, we have $n = \left\lceil \frac{d}{m} \right\rceil \leq d - 1$.
Hence $nk + 1 \leq (d-1)k + 1 < N$.
Since all the coordinates of the points in $\polytope{S}_{\b a}$ are nonnegative, this implies that $\b q_i \neq \b v_d$ for all $1\leq i \leq n$.
By \Cref{coro:No_lattice_point}, the lattice points of $\polytope{S}_{\b a}$ that are not vertices are precisely $\b p_i = (i, \ldots, i)$ for $1\leq i \leq k$. 
Therefore each $\b q_i$ is either a vertex $\b v_0, \ldots, \b v_{d-1}$ or a point $\b p_i$.
It follows that  $\b q_i \leq \b p_k$ coordinatewise for all $1\leq i \leq n$. 
Thus $\b x = \b q_1 + \cdots + \b q_n \leq (nk, \ldots, nk) < (nk + 1, \ldots, nk + 1) = \b x$ which is a contradiction, and the result follows.

\end{proof}

\begin{remark}
\textbf{Intuition behind \Cref{thm:not_IDP_2}.} 
The failure of IDP when $N = kd + m$ with $2 \leq m \leq d-1$ can be understood
as follows.  Collinear lattice points $\b p_i=(i,\ldots,i)$ also occur in these
cases, but their position relative to the base facet does not produce the
unimodular layering available when $N=kd$ or $N=kd+1$.  In particular, the
point $\b x = (nk+1, \ldots, nk+1)$ (see \Cref{lem:kS_dilation}) cannot be properly 
decomposed.
\end{remark}

We illustrate the failure of the integer decomposition property in \Cref{thm:not_IDP_2} in an example. 
\begin{example}
Let $d=3$, $k=2$, and $m=2$, so that $N=8$ and $n=\lceil d/m \rceil=2$.
For the Hermite normal form simplex $\polytope{S}_{\b a}$ associated with $\b a=(7,7,8)$, we have $\b x=(5,5,5)\in 2\polytope{S}_{\b a}$.
Observe that
$$
\polytope{S}_{\b a} \cap \Z^3 = \{(0,0,0),\ (1,0,0),\ (0,1,0),\ (1,1,1),\ (2,2,2),\ (7,7,8)\}.
$$
Therefore, there are no lattice points $\b y,\b z\in\polytope{S}_{\b a}$ such that $\b x=\b y+\b z$, and hence $\polytope{S}_{\b a}$ is not IDP.
\end{example}

The next statement is the remaining part of \Cref{thm:A} (3).
\begin{theorem}\label{thm:Not_IDP_1}
Let $1<N<d$ and let $\polytope{S}_{\b a}$ be the one-row Hermite normal form simplex associated with  $\b a = (N - 1, \dots, N - 1, N)\in \N^d$. 
Then $\polytope{S}_{\b a}$ does not satisfy the integer decomposition property.
\end{theorem}

\begin{proof}
Let $\polytope{S}_{\b a}^+ \coloneqq \{(1,\b x) ~:~ \b x \in \polytope{S}_{\b a}\} \subset \mathbb{R}^{d+1}$
be the homogenization of $\polytope{S}_{\b a}$.
Let $A$ be the matrix whose columns are the lattice points of
$\polytope{S}_{\b a}^+$. 
Since, by \Cref{coro:No_lattice_point}, the only lattice points of $\polytope{S}_{\b a}^+$
are its vertices, $A$ is a $(d+1)\times(d+1)$ matrix.
As $\det(A)\neq 0$, we have $\ker(A)=\{\b 0\}$, which directly implies that  the toric ideal $\mathcal{I}_{\polytope{S}_{\b a}} \coloneqq \langle x^{\b u^+} - x^{\b u^-} : \b u \in \ker_{\Z}(A) \rangle$
is trivial. Since it is well-known that  $\mathbb{K}[x_{\alpha}~:~\alpha\in \polytope{S}_{\b a}\cap \mathbb{Z}^d]/\mathcal{I}_{\polytope{S}_{\b a}}$ and $A_{\polytope{S}_{\b a}}$ have  the same Hilbert series, this implies
$$
\mathrm{Hilb}(A_{\polytope{S}_{\b a}};t) = \frac{1}{(1-t)^{d+1}}.
$$

On the other hand, the Ehrhart series of $\polytope{S}_{\b a}$ equals $\text{Ehr}_{\polytope{S}_{\b a}}(t) = \frac{h_{\polytope{S}_{\b a}}^\ast(t)}{(1 - t)^{d+1}}$. 
Since $h^\ast(1)=\nVol(\polytope{S}_{\b a})=N>1$, we have $h^\ast(t)\neq1$.
This implies that $\mathrm{Hilb}(A_{\polytope{S}_{\b a}};t)   \neq \text{Ehr}_{\polytope{S}_{\b a}}(t)$. 
It follows that $\polytope{S}_{\b a}$ does not have the integer decomposition property. 

\end{proof}
We note that combining \Cref{thm:Unimodular_triangulation},  \Cref{thm:Not_IDP_1} and \Cref{thm:not_IDP_2} yields \Cref{thm:A}.



\Cref{thm:A} is of particular interest, as the study of the integer decomposition property is a central topic connecting commutative algebra and the theory of lattice polytopes. 
In general, determining whether a lattice polytope satisfies the IDP is a difficult and widely open problem, even within restricted families. 
In the specific case of one-row Hermite normal form simplices, only a few results are known. 
It was shown in \cite[Proposition~3.8]{Bajo2025LocalF} that if $\b a_q = (q^{k-1},\ldots,q,1,q^k)$, then the associated simplex $\polytope{S}_{\b a_q}$ does not satisfy the IDP. 
On the other hand, the sequence $\b a_N = (1,\ldots,1,N)\in \mathbb{Z}^d$ is also studied in \cite{Bajo2025LocalF}, but without addressing the IDP property. We will now close this gap. We first need a preparatory lemma.

\begin{lemma}\label{pro:lattice_points_1N}
Let $\b a_N = (1,\ldots, 1, N)$, and let $\polytope{S}_{\b a_N}$ be the Hermite normal form simplex associated with $\b a_N$. 
If $d > 2$, then for all $N \in \N_{>1}$, the simplex $\polytope{S}_{\b a_N}$ contains only its vertices as lattice points.  
\end{lemma}

\begin{proof}
Let $\b x \in \polytope{S}_{\b a_N} \cap \Z^d$ be a lattice point. Assume by contradiction that $\b x\neq \b v_i$ for $0 \leq i \leq d$. Then there exist $\lambda_j \in [0,1)$ ($0\leq j\leq d$) with $\sum_{j = 0}^d \lambda_j = 1$ such that   $\b x = \sum_{j = 0}^d \lambda_j  \b v_j$. 
If $\b x = (x_1, \dots, x_d)$, the convex combination gives:
$$
x_j = \lambda_j + \lambda_d \quad \text{for } 1 \leq j \leq d - 1, \quad \text{and} \quad x_d = N \lambda_d.
$$
Thus, we have $x_j\in\{0,1\}$ for all $1\leq j\leq d-1$. Moreover, since the only lattice points of the simplex $\conv(\b v_0,\dots,\b v_{d-1})$ are its vertices, we must have  $\lambda_d>0$. This implies $x_j=1$ for all $1\leq j\leq d-1$, \ie $\b x=(1,\dots,1,N\lambda_d)$.
Therefore,
$$
1 = \lambda_0 + \sum_{j = 1}^{d-1} \lambda_j + \lambda_d = \lambda_0 + (d - 1) (1 - \lambda_d) + \lambda_d.
$$
Simplifying, we get $ \lambda_0 = (d - 2) (\lambda_d - 1)$.
Since $d > 2$ and $ 0\leq \lambda_d < 1 $, it follows that $\lambda_0 <0$, which is a contradiction.
Therefore, the only lattice points in $\polytope{S}_{\b a_N}$ are its vertices.
\end{proof}

Finally, using the same argument as in \Cref{thm:Not_IDP_1}, together with \Cref{pro:lattice_points_1N}, we obtain the following result.

\begin{proposition}\label{prop:NO_IPD_1N}
Let $\polytope{S}_{\b a_N}$ be the one-row Hermite normal form simplex associated with the vector $\b a_N = (1, \dots, 1, N)\in \Z^d$, where $N \in \N_{>1}$.  
If $d>2$, then $\polytope{S}_{\b a_N}$ does not satisfy the integer decomposition property.
\end{proposition}
\begin{proof}
The proof is identical to the proof of \Cref{thm:Not_IDP_1}, using \Cref{pro:lattice_points_1N} in place of \Cref{coro:No_lattice_point}.
\end{proof}

Note that the simplices associated with the two sequences studied in \cite{Bajo2025LocalF} fail to satisfy the integer decomposition property and consequently do not admit a unimodular triangulation.
This contrasts with the family $\b a=(N-1,\ldots,N-1,N)$ considered in the present work, for which we obtain an exact characterization of when the associated Hermite normal form simplex admits a unimodular triangulation.

\begin{question}
Can one find effective criteria on $\b a \in \N^d$ ensuring that the one-row Hermite normal form simplex $\polytope{S}_{\b a}$ admits a unimodular triangulation or satisfies the integer decomposition property?
\end{question}

A natural first step is to restrict to one-row vectors whose first $d-1$ entries are constant. 

\begin{question}
For which values of $q$ and $N$, with $1\leq q<N$, does the one-row Hermite normal form simplex associated with 
$\b a=(N-q,\ldots,N-q,N)$ admit a unimodular triangulation or satisfy the integer decomposition property?
\end{question}
Note that the family studied in this section corresponds to the case $q=1$ in the previous question.
\section{Ehrhart-theoretic properties}
\label{sect:1row}

Ehrhart positivity and the unimodality of the Ehrhart polynomial are classical problems in Ehrhart theory. 
In \cite{Solus_Fu_Unimodal}, the authors provide a comprehensive overview of results on these questions for various families of lattice polytopes.
In this section, we extend this overview to the case of one-row Hermite normal form simplices. 
More precisely, for the family associated with $\b a = (N - 1, \dots, N - 1, N)\in\N^d$ where $N = kd$ or $N = kd + 1$, we prove that these simplices are Ehrhart positive. 
Nevertheless, we observe that for certain positive integers $k$, the corresponding Ehrhart polynomial is not unimodal.

Two central objects in our study are the $h^{\ast}$-polynomial and the local $h^{\ast}$-polynomial. 
In \cite{Bajo2025LocalF} and \cite{Hibi_Hermite}, the authors studied the local $h^{\ast}$-vector and the $h^{\ast}$-vector for one-row Hermite normal form simplices, respectively, with particular emphasis on the unimodality of the local $h^{\ast}$-vector and the shifted symmetry property of the $h^{\ast}$-vector. 
In this section, we explicitly compute the $h^{\ast}$-vector using combinatorial techniques and provide a closed formula for the local $h^{\ast}$-polynomial in the cases mentioned above. In particular, we prove \Cref{thm:B}.

\subsection{Useful identities}
\begin{lemma}\label{lem:Sum_binomial}
Let $d\in\Z_{\geq1}$ and $t$ be an indeterminate. 
Then the following identity holds

\begin{equation}
\sum_{j = 0}^{d-1}\binom{t + j}{d} = \binom{t + d}{d + 1} - \binom{t}{d + 1}
\end{equation}
\end{lemma}
\begin{proof}
This follows directly by substituting $ \binom{t + j}{d} = \binom{t + j + 1}{d + 1} - \binom{t + j}{d + 1}$ in the sum and then observing that all but the first and the last summand of the sum cancel. 
\end{proof}

For $1 \le i \le d$, define the \defn{elementary symmetric polynomial} $e_i(X_1,\dots,X_d)$ of degree $i$ in the variables $X_1, \dots,X_d$ as
$$
e_i(X_1,\dots,X_d)
=
\sum_{1 \le j_1 < \dots < j_i \le d}
X_{j_1}\cdots X_{j_i}.
$$

It is well known (see \cite[p.~20]{macdonald1998symmetric}) that these polynomials satisfy the recurrence relation 
\begin{equation}\label{eq:Elementary_recurrence}
e_i(X_1,\dots,X_d) = e_i(X_1,\dots,X_{d-1}) + X_d e_{i-1}(X_1,\dots,X_{d-1}),
\end{equation}
with $e_0(X_1,\dots,X_d) = 1$ and $e_i(X_1,\dots,X_d) = 0$ if $i<0$.

\begin{lemma}\label{lem:Binomial_difference}
Let
\[
f(t)=\binom{t+d}{d}-\binom{t}{d}
\qquad\text{and}\qquad
g(t)=\binom{t+d}{d+1}-\binom{t}{d+1},
\]
where $d\in\Z_{>0}$. Then $f(t)$ has degree $d-1$ and all its
coefficients are strictly positive.  Moreover, the coefficient of $t^j$ in $g(t)$ is strictly positive when
$1\leq j\leq d$ and $j\equiv d\pmod 2$, and is zero otherwise. 
\end{lemma}

\begin{proof}
Let $S=\{1,\ldots,d\}$ and $T=\{1,\ldots,d-1\}$. Then
\[
d!f(t)
=
\prod_{j=1}^d(t+j)
-
t\prod_{j=1}^{d-1}(t-j).
\]
Writing both products in terms of elementary symmetric polynomials gives
\[
d!f(t)
=
\sum_{i=0}^d
\left(e_i(S)-(-1)^ie_i(T)\right)t^{d-i},
\]
where $e_d(T)=0$ and $e_i(S)$ and $e_i(T)$ mean evaluations of the symmetric polynomials. The coefficient of $t^d$ vanishes. If
$1\leq i\leq d-1$ is odd, then
\[
e_i(S)-(-1)^ie_i(T)=e_i(S)+e_i(T)>0.
\]
If $2\leq i\leq d-1$ is even, then using \eqref{eq:Elementary_recurrence} and $S=T\sqcup\{d\}$ we obtain
\[
e_i(S)-e_i(T)=d\,e_{i-1}(T)>0.
\]
Finally, the constant term is $e_d(S)=d!>0$. Hence $f(t)$ has degree
$d-1$ and all its coefficients are strictly positive.

For $g(t)$, set
\[
P(t)\coloneqq\prod_{j=0}^d(t+j).
\]
Then
\[
\prod_{j=0}^d(t-j)=(-1)^{d+1}P(-t),
\]
and therefore
\[
(d+1)!g(t)=P(t)-(-1)^{d+1}P(-t).
\]
All coefficients of $P(t)$ are nonnegative. The expression on the right
retains twice the coefficients of $P(t)$ in degrees having the same parity as
$d$ and cancels all coefficients in the remaining degrees. Thus every
coefficient of $g(t)$ is nonnegative.
\end{proof}

\subsection{The $h^{\ast}$- and the local $h^{\ast}$-polynomial}

Let $\polytope{S}_{\b a}$ be the one-row Hermite normal form simplex associated to $\b a = (N - 1, \ldots, N - 1, N) \in \N^d$.
The $h^{\ast}$-vector of $\polytope{S}_{\b a}$ was computed in \cite[Prop.~3.4]{Hibi_Hermite} for the case $N = dk$, where $k$ is a positive integer.
The goal of this section is to provide another proof for this statement  by applying \Cref{thm:hstar_collinear_cone} to  the unimodular triangulation established in \Cref{thm:Unimodular_triangulation}. Moreover, we will extend this result to $N=dk+1$ and prove \Cref{thm:A}. To compute the local $h^{\ast}$-polynomial of $\polytope{S}_{\b a}$ we will rely on \eqref{eq:local_h_S}.

We recall the first part of \Cref{thm:B}.
\begin{theorem}\label{thm:h_vector_OneR}

Let $\polytope{S}_{\b a}$ be the one-row Hermite normal form simplex associated with 
$\b a=(N-1,\ldots,N-1,N)$. Then
$$
h^\ast_{\polytope{S}_{\b a}}(t)=
\begin{cases}
1+k\displaystyle\sum_{j=1}^{d-1}t^j+(k-1)t^d, & \text{if } N=dk,\\[0.3em]
1+k\displaystyle\sum_{j=1}^{d}t^j, & \text{if } N=dk+1.
\end{cases}
$$
\end{theorem}

\begin{proof}
Let $\polytope{P}_0$ be the facet of $\polytope{S}_{\b a}$ opposite to $\b v_0$. By \Cref{thm:Unimodular_triangulation}, the simplex $\polytope{S}_{\b a}$ admits a unimodular collinear cone triangulation,
If $N=dk$ resp. $N=dk+1$, the simplex $\polytope{P}_0$ has a unique resp. no  lattice point in its relative interior. Hence,
 $$h^\ast_{\polytope{P}_0}(t)=\begin{cases}1+t+\cdots+t^{d-1}, \quad &\text{if } N=dk\\
 1, \quad &\text{if } N=dk+1
 \end{cases}.$$
 Applying \Cref{thm:hstar_collinear_cone}  with $m = k$ and $m=k+1$, respectively, gives
$$
h^\ast_{\polytope{S}_{\b a}}(t)=\begin{cases}
    
h^\ast_{\polytope{P}_0}(t)+(k-1)(t+t^2+\cdots+t^d)=1+k\sum_{j=1}^{d-1}t^j+(k-1)t^d,  \quad &\text{if } N=dk\\
h^\ast_{\polytope{P}_0}(t)+k(t+t^2+\cdots+t^d)=1+k\sum_{j=1}^d t^j,  \quad &\text{if } N=dk+1.
\end{cases}
$$

\end{proof}

For the rest of this section, we focus on the local $h^\ast$-polynomial of $\polytope{S}_{\b a}$. 
We recall the formula for the local $h^\ast$-polynomial from \Cref{thm:B}.

\begin{theorem}\label{thm:Local_one_row_dk}
 Let $\polytope{S}_{\b a}$ be the one-row Hermite normal form simplex associated with the vector $\b a = (N - 1, \dots, N - 1, N)\in \N^d$. 
 Then
 $$
\ell^\ast_{\polytope{S}_{\b a}}(t)=
\begin{cases}
(k-1)\displaystyle\sum_{i=1}^{d}t^i, & \text{if } N=dk,\\[0.3em]
k\displaystyle\sum_{i=1}^{d}t^i, & \text{if } N=dk+1.
\end{cases}
$$

\end{theorem}

\begin{proof}
First assume $N=dk$. 
Let $A$ be the matrix whose rows are $(1,\b v_i)$, where
$\b v_0,\ldots,\b v_d$ are the vertices of $\polytope{S}_{\b a}$.
We use \eqref{eq:local_h_S} to compute the local $h^\ast$-polynomial. First, as easy computation shows that the last row $\overline{\b a}$ of $A^{-1}$ is given by
\[
\overline{\b a}
=
\left(
\frac{m(N-1-k)}{N},
\frac{(1-N)}{N},
\ldots,
\frac{(1-N)}{N},
\frac{1}{N}
\right)
\]
which directly implies that the elements of $\Gamma(\polytope{S}_{\b a})$ are of the form

\[
\b x_m= 
m\overline{\b a}
=
\left(
\frac{md(N-1-k)}{N},
\frac{m(1-N)}{N},
\ldots,
\frac{m(1-N)}{N},
\frac{m}{N}
\right) \quad \text{for } 0\leq m\leq N-1.
\]
Since \(N=dk\), we have
\[
\frac{d(N-1-k)}{N}
=
\frac{k-1}{k}+d-2
\qquad\text{and}\qquad
\frac{1-N}{N}
=
\frac1N-1.
\]
Using that $0\leq m\leq N-1$  we conclude that  
\[
\b x_m
=
\left(
\operatorname{frac}\left(\frac{m(k-1)}{k}\right),
\frac{m}{N},
\ldots,
\frac{m}{N}
\right).
\]

If \(k\mid m\), then
$
\operatorname{frac}\left(\frac{m(k-1)}{k}\right)=0$, 
so \(\b x_m\notin(0,1)^{d+1}\) and \(\b x_m\) does not contribute to the local
\(h^\ast\)-polynomial.

If \(k\nmid m\), there exist $0\leq q\leq d-1$ and $1\leq r\leq k-1$ such that $m=qk+r$ and it follows that 
$
\operatorname{frac}\left(\frac{m(k-1)}{k}\right)
=
\operatorname{frac}\left(\frac{r(k-1)}{k}\right)
=
1-\frac{r}{k}$. 
Combining the previous arguments, an easy computation shows
\[
\operatorname{age}(\b x_m)
=
1-\frac{r}{k}
+
d\frac{m}{N}
=
q+1.
\]
Finally, using \eqref{eq:local_h_S}, we obtain 
 
\[
\ell^\ast_{\polytope{S}_{\b a}}(t)
=
\sum_{q=0}^{d-1}(k-1)t^{q+1}
=
(k-1)\sum_{i=1}^d t^i,
\]
as claimed.

If $N=dk+1$, then \Cref{thm:h_vector_OneR} implies that $h^{\ast}_{\polytope{S}_{\b a}}(t)$ is shifted symmetric and the claim directly follows  from \cite[Prop. 2.17]{Bajo2025LocalF}.
\end{proof}

\subsection{The Ehrhart polynomial}

Let $\polytope{S}_{\b a}$ be the one-row Hermite normal form simplex associated to $\b a = (N - 1, \ldots, N - 1, N) \in \N^d$. 
In this section, we compute explicitly the Ehrhart polynomial of $\polytope{S}_{\b a}$ when $N\in\{dk,dk+1\}$, and analyze its positivity and unimodality properties.

\begin{theorem}\label{thm:Ehrhar_one_row}
Let $\polytope{S}_{\b a}$ be the one-row Hermite normal form simplex associated with $\b a = (N - 1, \dots, N - 1, N)\in\N^d$.  
If $N = d k$ or $N = d k + 1$ for some $k \in \N$,
then $\polytope{S}_{\b a}$ is Ehrhart positive.
\end{theorem}

\begin{proof}
First, suppose that $N = d k + 1$.  
Using \Cref{thm:h_vector_OneR}, and the usual expansion of the Ehrhart polynomial in the binomial basis (see e.g., \cite[Lemma 3.14]{beck2015computing}) we get

\begin{align*}
    \mathcal{L}_{\polytope{S}_{\b a}}(t) = \binom{t + d}{d} + k \sum_{j = 0}^{d-1} \binom{t + j}{d} 
    = \binom{t + d}{d} + k \left( \binom{t + d}{d + 1} - \binom{t}{d + 1} \right),
\end{align*}
where the second equality follows from \Cref{lem:Sum_binomial}.
As $\binom{t + d}{d}$ is a polynomial of degree $d$ with strictly positive coefficients and as  $\binom{t + d}{d + 1} - \binom{t}{d + 1}$ is a polynomial of degree $d$ with nonnegative coefficients by \Cref{lem:Binomial_difference}, it follows that all coefficients of $\mathcal{L}_{\polytope{S}_{\b a}}(t)$ are positive, and hence $\polytope{S}_{\b a}$ is Ehrhart positive.

Let $N = d k$ for some positive integer $k$.  By the same arguments as in the previous case it follows that 
\begin{align*}
\mathcal{L}_{\polytope{S}_{\b a}}(t)
    = \binom{t + d}{d} - \binom{t}{d} + k \left( \binom{t + d}{d + 1} - \binom{t}{d + 1} \right).
\end{align*}
The claim follows again by applying \Cref{lem:Binomial_difference}.
\end{proof}

\begin{remark}\label{rem:Ehrhart_positive_not_unimodal}
By \Cref{lem:Binomial_difference}, the difference $\binom{t + d}{d + 1} - \binom{t}{d + 1}$ is a polynomial of degree $d$, and if $m$ is an odd integer with $1\leq m \leq d$, then the coefficient $c_{d - m}$ of $t^{d - m}$ in this polynomial satisfies $c_{d - m} = 0$.
This implies that, for $\b a = (N - 1, \dots, N - 1, N)$ with $N \in\{d k,dk+1\}$, the coefficient $ \ell_{d - j}$ of $t^{d-j}$ in  the Ehrhart polynomial $\mathcal{L}_{\polytope{S}_{\b a}}(t)$ is independent of $k$ whenever $j$ is odd and $1 \leq j \leq d$.
\end{remark}

we illustrate this behavior in an example. 
\begin{example}\label{ex:Ehrhart_positive}
Let $\b a = (5, 5, 5, 5, 5, 6)$ and let $\polytope{S}_{\b a}$ be the Hermite normal form simplex associated with $\b a$, where $N = 6$.  
Then the Ehrhart polynomial equals
$$
\mathcal{L}_{\polytope{S}_{\b a}}(t) = \frac{1}{120} t^6 + \textcolor{blue}{\frac{1}{20}} t^5 + \frac{5}{12} t^4 + \textcolor{blue}{\frac{4}{3}} t^3 + \frac{103}{40} t^2 + \textcolor{blue}{\frac{157}{60}} t + 1.
$$

Now, consider $\b a' = (41, 41, 41, 41, 41, 42)$ and let $\polytope{S}_{\b a'}$ be the associated Hermite normal form simplex with $N = 42$.  
Then the Ehrhart polynomial equals
$$
\mathcal{L}_{\polytope{S}_{\b a'}}(t) = \frac{7}{120} t^6 + \textcolor{blue}{\frac{1}{20}} t^5 + \frac{13}{6} t^4 + \textcolor{blue}{\frac{4}{3}} t^3 + \frac{271}{40} t^2 + \textcolor{blue}{\frac{157}{60}} t + 1.
$$

In both cases, the coefficients of $t^5$, $t^3$, and $t$ are the same.
\end{example}


We now show that though  $\polytope{S}_{\b a}$ is Ehrhart positive it is not unimodal, if $k$ is large enough, where again $\b a=(N-1,\ldots,N-1,N)$ with $N\in \{kd,kd+1\}$. 
\begin{theorem}\label{thm:explicit_non_unimodality_bound}
Let $d\geq3$, and let 
$h(t)\coloneqq\binom{t+d}{d}=\sum_{j=0}^d h_jt^j$, $f(t)\coloneqq\binom{t+d}{d}-\binom{t}{d}=\sum_{j=0}^d f_jt^j$ and $g(t)\coloneqq\binom{t+d}{d+1}-\binom{t}{d+1}=\sum_{j=0}^d g_jt^j$. 
For $m\in\{f,h\}$, set
$$
K_m(d)\coloneqq
\max\left\{1,\,
1+\left\lfloor
\max\left\{
\frac{\ell_{r+1}-\ell_r}{g_r},
\frac{\ell_{r+1}-\ell_{r+2}}{g_{r+2}}
\right\}
\right\rfloor\right\}
$$
where $r=1$ if $d$ is odd, and $r=2$ if $d$ is even. 
Let $\polytope{S}_{\b a}$ be the one-row Hermite normal form simplex  associated with
$\b a=(N-1,\ldots,N-1,N)$. If $N=dk+1$ and $N=dk$, respectively, and $k\geq K_h(d)$ and $K_f(d)$, respectively, then  Ehrhart polynomial of $\polytope{S}_{\b a}$ is not unimodal.
\end{theorem}

\begin{proof}
By \Cref{lem:Binomial_difference}, we have $g_r>0$, $g_{r+1}=0$, and
$g_{r+2}>0$ for the stated choice of $r$. If $N=dk+1$, it follows from  the proof of
\Cref{thm:Ehrhar_one_row} that
$$
\mathcal L_{\polytope{S}_{\b a}}(t)=h(t)+kg(t).
$$
Writing $\mathcal L_{\polytope{S}_{\b a}}(t)=\sum_{j=0}^d\ell_jt^j$, the
definition of $K_h(d)$ implies
$$
\ell_r=h_r+kg_r>h_{r+1}=\ell_{r+1}
<h_{r+2}+kg_{r+2}=\ell_{r+2},
$$
which shows that $\mathcal L_{\polytope{S}_{\b a}}(t)$ is not unimodal.  If $N=dk$, the same proof applied to 
$$
\mathcal L_{\polytope{S}_{\b a}}(t)=f(t)+kg(t),
$$
with $K_f(d)$ in place of $K_h(d)$ shows the claim.
\end{proof}

\begin{remark}\label{rem:Non_unimodality_one_row}
The bounds in \Cref{thm:explicit_non_unimodality_bound} are explicit sufficient
bounds and are not asserted to be optimal.  The two families can have different
optimal thresholds; for example, in dimension $3$ the first non-unimodal cases
occur at $k=4$ for $N=3k$ and at $k=2$ for $N=3k+1$.
\end{remark}
%

\section{Two-row Hermite normal form simplex}
\label{sect:2row}

In this section, we extend the analysis from the precious sections to two-row Hermite normal form simplices. 
We consider sequences $\b b = (M-1, \ldots, M-1, M, 0)\in \N^d$ and $\b a = (1, \ldots, 1, N)\in \N^d$, and prove \Cref{thm:C}, which provides conditions on $N$ and $M$ that guarantee the existence of a unimodular triangulation. 
For $M\leq d$, we characterize completely when these simplices have the integer decomposition property.  In the positive cases we
compute both the $h^{\ast}$-polynomial and the local $h^{\ast}$-polynomial and
study Ehrhart positivity and unimodality.

\begin{definition}
We say that a simplex $\polytope{S}$ is in \defn{two-row Hermite normal form} if its Hermite normal form matrix $H$ satisfies $a_{i,i} = 1$ for all $1 \leq i \leq d-2$.
In this case, we refer to such a simplex using only the last  two rows of $H$: 
$$
\b b = (b_1, \dots, b_{d-2}, M, 0) \quad \text{and} \quad \b a = (a_1, \dots, a_{d-2}, a_{d-1}, N),
$$
where $a_i\coloneqq a_{d,i} = a_i$, $a_{d,d} = N$, $b_i\coloneqq b_{d-1,i}$,  and $b_{d-1,d-1}=M $. We will write $\polytope{S}_{\b a, \b b}$ for the corresponding Hermite normal form simplex.
\end{definition}
Since, by definition, the  vertices of $\polytope{S}_{\b a, \b b}$ are given by the rows of the matrix $H$, we have
$$
\polytope{S}_{\b a, \b b} = \conv\left\{
\underbrace{(0, \dots, 0)}_{\b v_0},\,
\underbrace{(1, 0, \dots, 0)}_{\b v_1},\,
\dots,\,
\underbrace{(b_1, \dots, M, 0)}_{\b v_{d-1}},\,
\underbrace{(a_1, \dots, a_{d-1}, N)}_{\b v_d}
\right\}.
$$
It is straightforward, that  $\polytope{S}_{\b a, \b b}$ has normalized volume equal to $M\times N$ (see also \cite[Proposition $2.5$]{Bajo2025LocalF}).

\begin{example}\label{ex:Matrix_2H_A}
Suppose $\b a = (2, 3, 4, 4, 5)$ and $\b b = (1, 4, 3, 7, 0)$, so that $N = 5$ and $M = 7$.
Then, the corresponding Hermite normal form matrix is

\begin{center}
$H = \begin{pmatrix}
0 & 0 & 0 & 0 & 0 \\
1 & 0 & 0 & 0 & 0 \\
0 & 1 & 0 & 0 & 0 \\
0 & 0 & 1 & 0 & 0 \\
1 & 4 & 3 & 7 & 0 \\
2 & 3 & 4 & 4 & 5
\end{pmatrix}$
\end{center}
\end{example}

We now state a simple lemma that will be helpful.

\begin{corollary}\label{coro:Binomial_difference_2}
Let $d \geq 2$ be an integer and let $\tilde{f}(t) = \binom{t + d}{d} - \binom{t + 1}{d}$ and $\tilde{g}(t) = \binom{t + d}{d + 1} - \binom{t + 1}{d + 1}$ be polynomials in the variable $t$.
$\tilde f(t)$ has degree $d-1$ and all its coefficients are strictly positive. 
Moreover, $\tilde g(t)$ has degree $d$, zero constant term, and all its remaining coefficients are strictly positive.
\end{corollary}

\begin{proof}
First note that, by definition, we have
\[
d!\tilde f(t)
=
\prod_{j=1}^d(t+j)
-
\prod_{j=-1}^{d-2}(t-j).
\]
Set \(U=\{1,\ldots,d-2\}\). Using
\begin{equation}\label{sym}
e_i(U\cup\{x,y\})
=
e_i(U)+(x+y)e_{i-1}(U)+xy\,e_{i-2}(U),
\end{equation}
the coefficient of \(t^{d-i}\) in \(d!\tilde f(t)\) equals
\[
\begin{cases}
2d\,e_{i-1}(U)+d(d-1)e_{i-2}(U),
    & \text{ if } i \text{ is even},\\[2mm]
2e_i(U)+(2d-2)e_{i-1}(U)+d(d-1)e_{i-2}(U),
    & \text{ if }i \text{ is odd}.
\end{cases}
\]
For \(i=0\), this coefficient is zero, while it is strictly positive for
\(1\leq i\leq d\). Hence \(\tilde f(t)\) has degree \(d-1\) and all its
coefficients are strictly positive.

Similarly, we have
\[
(d+1)!\tilde g(t)
=
\prod_{j=0}^d(t+j)
-
\prod_{j=-1}^{d-1}(t-j).
\]
Set \(V=\{1,\ldots,d-1\}\). Using \eqref{sym} with $V$ instead of $U$, we get that the coefficient of \(t^{d+1-i}\) in \((d+1)!\tilde g(t)\) equals
\[
\begin{cases}
(d+1)e_{i-1}(V),
    & \text{ if }i \text{ is even},\\[2mm]
2e_i(V)+(d-1)e_{i-1}(V),
    & \text{ if }i \text{ is odd}.
\end{cases}
\]
This coefficient vanishes for \(i=0\) and \(i=d+1\), and is strictly
positive for \(1\leq i\leq d\). Therefore, \(\tilde g(t)\) has degree \(d\),
zero constant term, and all its remaining coefficients are strictly
positive.
\end{proof}


\subsection{The integer decomposition property for $\polytope{S}_{\b a, \b b}$}


The goal of this section is to provide a natural extension of the unimodular triangulation presented in \Cref{subsection_UT_1row}.
We start with the description of the lattice points in some special cases.

\begin{proposition}\label{prop:Interio_points_two}
Let $\b a = (1, \dots, 1,\, N)\in \mathbb{N}^d$, $\b b = (M-1, \ldots, M-1, M,0)\in \mathbb{N}^d$ and let $\polytope{S}_{\b a, \b b}$ be the Hermite normal form simplex associated with $\b a$ and $\b b$.  
If $M\in\{d-1,d\}$ and $N \in \Z_{>1}$, then $\polytope{S}_{\b a, \b b}$ contains exactly its vertices and the points $\b p_i = (1, \dots,1, i)$ for $0 \leq i \leq N - 1$ as lattice points.
\end{proposition}
Geometrically, the lattice points $\b p_i=(1,\ldots,1,i)$ of
$\polytope{S}_{\b a,\b b}$ lie on a distinguished line parallel to the last
coordinate axis.  We will use these points to build a triangulation via
successive coning.

\begin{proof}
Let $\b x = (x_1, \ldots, x_d)  \in \polytope{S}_{\b a, \b b}\cap \mathbb{Z}^d$ and assume that  $\b x \ne \b v_i$ for every vertex $\b v_i$ of $\polytope{S}_{\b a, \b b}$.
Then, writing $\b x$ as a convex combination $\b x = \sum_{i=0}^{d} \lambda_i \b v_i$ with $\sum_{i=0}^d \lambda_i = 1$ and $\lambda_i\geq 0$ for $0\leq i\leq d$, it follows that 
\begin{equation}\label{eq:IDP2}
x_j = \lambda_j + (M - 1)\lambda_{d-1} + \lambda_d \quad \text{for } 1 \leq j \leq d - 2,\quad
x_{d-1} = M\lambda_{d-1} + \lambda_d,\quad
x_d = N \lambda_d.
\end{equation}
This implies $x_j-x_{d-1}=\lambda_j-\lambda_{d-1}$ for  $1\leq j\leq d-2$. Since the left-hand side is an integer, whereas the right-hand side belongs to
\((-1,1)\), we conclude $\lambda_j=\lambda_{d-1}$ and $x_j=x_{d-1}$ for $1\leq j\leq d-2$. It follows that $1=
\lambda_0+(d-1)\lambda_{d-1}+\lambda_d$.

Suppose first that \(M=d-1\). Since $x_{d-1}=(d-1)\lambda_{d-1}+\lambda_d$, 
we obtain from the previous equation $1=\lambda_0+x_{d-1}$. 
As \(x_{d-1}\) is a nonnegative integer and \(0\leq\lambda_0<1\), this forces $x_{d-1}=1$ and $\lambda_0=0$.

Now suppose that \(M=d\). In this case, we get $1
=
\lambda_0+x_{d-1}-\lambda_{d-1}$,
and hence $x_{d-1}=1-\lambda_0+\lambda_{d-1}$.
As the right-hand side lies strictly between \(0\) and \(2\), while
\(x_{d-1}\) is an integer, we conclude $x_{d-1}=1$ and $\lambda_0=\lambda_{d-1}$.

In both cases, we have $x_1=\cdots=x_{d-1}=1$ and $x_d=N\lambda_d$. Since \(0\leq\lambda_d<1\) and $ x_d\in\mathbb{Z} $, it follows that 
$0\leq x_d\leq N-1$, which implies that $\b x$ is of the form $\b p_i=(1,\ldots,1,i)$, for $0\leq i\leq N-1$. 
It remains to show that $\b p_i\in \polytope{S}_{\b a,\b b}$.
It is easy to check that $\b p_0
=
\frac{1}{d-1}\sum_{j=1}^{d-1}\b v_j
$ and $\b p_0
=
\frac{1}{d}\sum_{j=0}^{d-1}\b v_j$, respectively,  if $M=d-1$ and $M=d$, respectively. 
Hence, \(\b p_0\in\polytope{S}_{\b a,\b b}\) in both cases. Moreover, 
$\b p_i
=
\left(1-\frac{i}{N}\right)\b p_0
+
\frac{i}{N}\b v_d$ for $0\leq i\leq N$ which shows $\b p_i\in \polytope{S}_{\b a,\b b}$.\end{proof}

\begin{remark}\label{remark:Interior_poins_2}
Suppose that $M = d - m$ for some $2\leq m \leq d-1$. 
We know that $ x_j \geq 1$ for all $1\leq j \leq d-1$.
It follows from the proof of \Cref{prop:Interio_points_two} that  $ 1 = \lambda_0 + (d - 1)\lambda_{d-1} + \lambda_d.$
Using $x_{d-1} = (d - m)\lambda_{d-1} + \lambda_d$, we can rewrite the equation as:
$$
1 = \lambda_0 + x_{d-1} + (m - 1)\lambda_{d-1}.
$$
Since $x_{d-1} \geq 1$, $\lambda_{d-1} \geq 0$ and $m \geq 2$, we conclude that $\lambda_0 < 0$, a contradiction.
Hence, no such lattice point $\b x$ can exist, and $\polytope{S}_{\b a, \b b}$ contains only its vertices as lattice points.
\end{remark}

We now recall \Cref{thm:C} from the introduction.
\begin{theorem}\label{thm_UT_2}
Let $\b a = (1, \dots, 1,\, N)\in \mathbb{Z}^d$ and $\b b = (M - 1, \dots, M - 1, M, 0)\in \mathbb{Z}^d$, and let $\polytope{S}_{\b a, \b b}$ be the Hermite normal form simplex associated with $\b a$ and $\b b$.  
If $M \in\{d - 1,d\}$, and $N \in \Z_{>1}$, then $\polytope{S}_{\b a, \b b}$ admits a regular and unimodular triangulation.
When $d=3$ and $M=d-1$, the triangulation can moreover be chosen flag.
\end{theorem}

\begin{proof}
By \Cref{prop:Interio_points_two}, the simplex $\polytope{S}_{\b a, \b b}$ contains no lattice points other than its vertices and the points $\b p_i=(1,\dots,1,i)$, $0\leq i\leq N-1$.

First assume that $M=d$.
Let $\polytope{F}\coloneqq\conv(\b v_0,\dots,\b v_{d-1})$. 
Since $\b p_0\in\relint(\polytope{F})$, the stellar subdivision of $\polytope{F}$ with center $\b p_0$ yields a regular triangulation $\mathcal T_F$ consisting of $M=d$ simplices of dimension $d-1$. 
Applying \Cref{def:colinear_cone_triangulation} to $\polytope{S}_{\b a, \b b}$ with distinguished vertex $\b v_d\coloneqq (1,\ldots,1,N)$ and base triangulation $\mathcal T_F$ produces a collinear cone triangulation $\mathcal T$ with exactly $N M$ maximal simplices of dimension $d$. 
By \Cref{remark:flag_regular}, the triangulation $\mathcal T$ is also regular.

Now assume that $M=d-1$.
Let $\polytope{Q}\coloneqq\conv(\b v_1,\dots,\b v_d)$ and let $\polytope{H}\coloneqq\conv(\b v_1,\dots,\b v_{d-1})$. 
By \Cref{prop:Interio_points_two}, 
$$\polytope{Q}\cap\Z^d=(\polytope{H}\cap\Z^d)\sqcup\{\b v_d,\b p_1,\dots,\b p_{N-1}\},$$ and $\b p_0\in\relint(\polytope{H})$. 
The stellar subdivision of $\polytope{H}$ with center $\b p_0$ yields a regular triangulation of $\polytope{H}$ with $M=d-1$ simplices of dimension $d-2$. 
Applying \Cref{def:colinear_cone_triangulation} to $\polytope{Q}$ with distinguished vertex $\b v_d$ along the geometric order induced by the line through $\polytope{H}$ and $\b v_d$ produces a regular triangulation of $\polytope{Q}$ consisting of exactly $N(d-1)=NM$ simplices of dimension $d-1$ by \Cref{remark:flag_regular}. 
Finally, coning this triangulation with $\b v_0$ yields a triangulation $\mathcal T$ of $\polytope{S}_{\b a, \b b}$ with exactly $N(d-1)=NM$ maximal simplices of dimension $d$.
This triangulation is regular: a lifting function inducing the regular triangulation of $\polytope Q$ extends to the pyramid by assigning the apex $\b v_0$ a sufficiently small height.
If $d=3$, then $\polytope H$ is a segment and its stellar subdivision is flag.
The iterated one-point extensions preserve flagness by
\Cref{remark:flag_regular}, and coning a flag complex preserves flagness.
Thus $\mathcal T$ is flag when $d=3$ and $M=d-1$.

Since, in both cases, we have $\nVol(\polytope{S}_{\b a, \b b})=NM$,  \Cref{lem:volume_count_unimodular} implies that the triangulation is unimodular.
\end{proof}


We illustrate the previous theorem with an example.
\begin{example}
Let $d = 3$ and let $\polytope{S}_{\b a,\b b}$ be the Hermite normal form simplex associated with $\b a = (1, 1, 5)$ and $\b b = (1, 2, 0)$, \ie

$$\polytope{S}_{\b a, \b b} = \conv\{ \underbrace{(0, 0, 0)}_{\b v_0},\, \underbrace{(1, 0, 0)}_{\b v_1},\, \underbrace{(1, 2, 0)}_{\b v_2},\, \underbrace{(1,1,5)}_{\b v_3 = \b p_5}\}.$$

We know that $\nVol(\polytope{S}_{\b a,\b b}) = 10$. The following simplices give a regular unimodular triangulation of $\polytope{S}_{\b a,\b b}$  (see \Cref{fig:triangulation_d_2} for an illustration):

\begin{multicols}{2}
$\mathcal{T}_{10} = \conv\{(0, 0, 0), (1, 0, 0), (1, 1, 0), (1, 1, 1)\}$

\vspace{0.4em}

$\mathcal{T}_9 = \conv\{(0, 0, 0), (1, 0, 0), (1, 1, 1), (1, 1, 2)\}$

\vspace{0.4em}

$\mathcal{T}_8 = \conv\{(0, 0, 0), (1, 0, 0), (1, 1, 2), (1, 1, 3)\}$

\vspace{0.4em}

$\mathcal{T}_7 = \conv\{(0, 0, 0), (1, 0, 0), (1, 1, 3), (1, 1, 4)\}$

\vspace{0.4em}

$\mathcal{T}_6 = \conv\{(0, 0, 0), (1, 0, 0), (1, 1, 4), (1, 1, 5)\}$

\vspace{0.4em}

$\mathcal{T}_5 = \conv\{(0, 0, 0), (1, 1, 0), (1, 1, 1), (1, 2, 0)\}$

\vspace{0.4em}

$\mathcal{T}_4 = \conv\{(0, 0, 0), (1, 1, 1), (1, 1, 2), (1, 2, 0)\}$

\vspace{0.4em}

$\mathcal{T}_3 = \conv\{(0, 0, 0), (1, 1, 2), (1, 1, 3), (1, 2, 0)\}$

\vspace{0.4em}

$\mathcal{T}_2 = \conv\{(0, 0, 0), (1, 1, 3), (1, 1, 4), (1, 2, 0)\}$

\vspace{0.4em}

$\mathcal{T}_1 = \conv\{(0, 0, 0), (1, 1, 4), (1, 1, 5), (1, 2, 0)\}$
\end{multicols}

\begin{figure}[ht]
\centering
\tdplotsetmaincoords{70}{70} 
\begin{tikzpicture}[tdplot_main_coords, scale=2.0]

\definecolor{myblue}{rgb}{0.6,0.8,1}
\definecolor{mygreen}{rgb}{0.65,0.9,0.65}
\definecolor{myred}{rgb}{1,0.7,0.7}
\definecolor{myyellow}{rgb}{1,1,0.6}
\definecolor{mypurple}{rgb}{0.85,0.75,1}

\begin{scope}[xshift=0cm]

\coordinate (v0) at (0, 0, 0);
\coordinate (v1) at (1, 0, 0);
\coordinate (v2) at (1, 2, 0);
\coordinate (v3) at (1, 1, 5);
\coordinate (p0)  at (1, 1, 0);
\coordinate (p1)  at (1, 1, 1);
\coordinate (p2)  at (1, 1, 2);
\coordinate (p3)  at (1, 1, 3);
\coordinate (p4)  at (1, 1, 4);

\filldraw[fill=mypurple, opacity=0.6] (v0) -- (v2) -- (p4) -- cycle;
\filldraw[fill=mypurple, opacity=0.6] (v0) -- (v2) -- (v3) -- cycle;
\filldraw[fill=mypurple, opacity=0.6] (v0) -- (p4) -- (v3) -- cycle;
\filldraw[fill=mypurple, opacity=0.6] (v2) -- (p4) -- (v3) -- cycle;

\filldraw[fill=myyellow, opacity=0.6] (v0) -- (v2) -- (p3) -- cycle;
\filldraw[fill=myyellow, opacity=0.6] (v0) -- (v2) -- (p4) -- cycle;
\filldraw[fill=myyellow, opacity=0.6] (v0) -- (p3) -- (p4) -- cycle;
\filldraw[fill=myyellow, opacity=0.6] (v2) -- (p3) -- (p4) -- cycle;

\filldraw[fill=myred, opacity=0.6] (v0) -- (v2) -- (p2) -- cycle;
\filldraw[fill=myred, opacity=0.6] (v0) -- (v2) -- (p3) -- cycle;
\filldraw[fill=myred, opacity=0.6] (v0) -- (p2) -- (p3) -- cycle;
\filldraw[fill=myred, opacity=0.6] (v2) -- (p2) -- (p3) -- cycle;

\filldraw[fill=mygreen, opacity=0.6] (v0) -- (v2) -- (p1) -- cycle;
\filldraw[fill=mygreen, opacity=0.6] (v0) -- (v2) -- (p2) -- cycle;
\filldraw[fill=mygreen, opacity=0.6] (v0) -- (p1) -- (p2) -- cycle;
\filldraw[fill=mygreen, opacity=0.6] (v2) -- (p1) -- (p2) -- cycle;

\filldraw[fill=myblue, opacity=0.6] (v0) -- (v2) -- (p1) -- cycle;
\filldraw[fill=myblue, opacity=0.6] (v0) -- (v2) -- (p0) -- cycle;
\filldraw[fill=myblue, opacity=0.6] (v0) -- (p1) -- (p0) -- cycle;
\filldraw[fill=myblue, opacity=0.6] (v2) -- (p1) -- (p0) -- cycle;

\filldraw[fill=myblue, opacity=0.6] (v0) -- (v1) -- (p1) -- cycle;
\filldraw[fill=myblue, opacity=0.6] (v0) -- (v1) -- (p0) -- cycle;
\filldraw[fill=myblue, opacity=0.6] (v0) -- (p1) -- (p0) -- cycle;
\filldraw[fill=myblue, opacity=0.6] (v1) -- (p1) -- (p0) -- cycle;

\filldraw[fill=mygreen, opacity=0.6] (v0) -- (v1) -- (p1) -- cycle;
\filldraw[fill=mygreen, opacity=0.6] (v0) -- (v1) -- (p2) -- cycle;
\filldraw[fill=mygreen, opacity=0.6] (v0) -- (p1) -- (p2) -- cycle;
\filldraw[fill=mygreen, opacity=0.6] (v1) -- (p1) -- (p2) -- cycle;

\filldraw[fill=myred, opacity=0.6] (v0) -- (v1) -- (p2) -- cycle;
\filldraw[fill=myred, opacity=0.6] (v0) -- (v1) -- (p3) -- cycle;
\filldraw[fill=myred, opacity=0.6] (v0) -- (p2) -- (p3) -- cycle;
\filldraw[fill=myred, opacity=0.6] (v1) -- (p2) -- (p3) -- cycle;

\filldraw[fill=myyellow, opacity=0.6] (v0) -- (v1) -- (p3) -- cycle;
\filldraw[fill=myyellow, opacity=0.6] (v0) -- (v1) -- (p4) -- cycle;
\filldraw[fill=myyellow, opacity=0.6] (v0) -- (p3) -- (p4) -- cycle;
\filldraw[fill=myyellow, opacity=0.6] (v1) -- (p3) -- (p4) -- cycle;

\filldraw[fill=mypurple, opacity=0.6] (v0) -- (v1) -- (p4) -- cycle;
\filldraw[fill=mypurple, opacity=0.6] (v0) -- (v1) -- (v3) -- cycle;
\filldraw[fill=mypurple, opacity=0.6] (v0) -- (p4) -- (v3) -- cycle;
\filldraw[fill=mypurple, opacity=0.6] (v1) -- (p4) -- (v3) -- cycle;

\foreach \pt/\name in {v0/$\b v_0$, v1/$\b v_1$, v2/$\b v_2$, v3/$\b v_3$, p0/$\b p_0$, p1/$\b p_1$, p2/$\b p_2$, p3/$\b p_3$, p4/$\b p_4$} {
  \filldraw[black] (\pt) circle (0.4pt) node[anchor=north east] {\name};
}
\end{scope}
\end{tikzpicture} 
\caption{Unimodular triangulation for the two-row Hermite normal form simplex $\polytope{S}_{\b a, \b b}$ associated with $\b a = (1, 1, 5)$ and $\b b = (1, 2, 0)$.}
\label{fig:triangulation_d_2}
\end{figure}
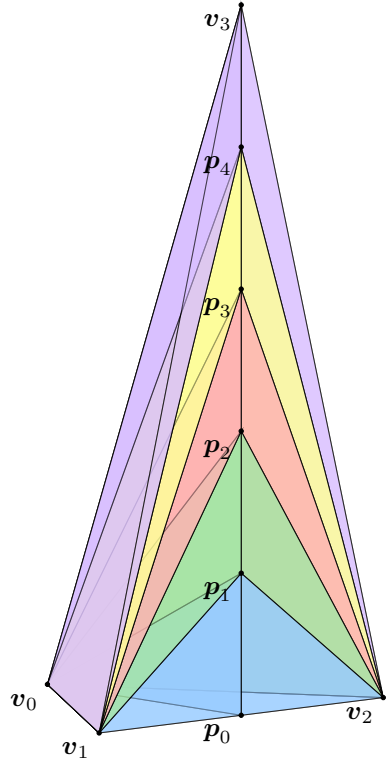
\end{example}

The triangulation defined in \Cref{thm_UT_2} is of particular interest because the one-row Hermite normal form simplex associated with $\b a = (1, \ldots, 1, N)$ does \textbf{not} satisfy the IDP for any $N \in \N_{>1}$ (see \Cref{prop:NO_IPD_1N}). 
In contrast, consider $\b a=(1,\ldots,1,N)$ and 
$\b b=(M-1,\ldots,M-1,M,0)$.
If $M\in\{d-1,d\}$ and $d>2$, then the corresponding two-row Hermite normal form simplex $\polytope{S}_{\b a,\b b}$ does satisfy the IDP. 
In other words, adding a second row $\b b = (M - 1, \ldots, M - 1, M, 0)$ with $M \in \{d - 1, d\}$ provides sufficient structure for the two-row Hermite normal form simplex $\polytope{S}_{\b a, \b b}$ to admit a unimodular triangulation, even when the last row is given by $\b a = (1, \ldots, 1, N)$ with $N \in \N_{>1}$.
On the other hand, if we consider $M < d - 1$, we have the following result.

\begin{proposition}\label{prop:non_idp_two_rows}
Let $\polytope{S}_{\b a,\b b}$ be the two-row Hermite normal form simplex associated with 
$\b a=(1,\ldots,1,N)$ and $\b b=(M-1,\ldots,M-1,M,0)$.
If $N\in\N_{\geq 2}$ and $M<d-1$, then $\polytope{S}_{\b a,\b b}$ does not satisfy the integer decomposition property.
\end{proposition}

\begin{proof}
By \Cref{remark:Interior_poins_2}, the only lattice points of
$\polytope{S}_{\b a,\b b}$ are its vertices. The same arguments as in the proof of \Cref{thm:Not_IDP_1} show that $\polytope{S}_{\b a,\b b}$ does not satisfy the integer decomposition property. 
\end{proof}

\begin{corollary}\label{coro:idp_classification_two_rows}
Let $1\leq M\leq d$, $N\geq2$, and let $\polytope{S}_{\b a,\b b}$ be the Hermite normal form simplex associated with $\b a=(1,\ldots,1,N)$ and $\b b=(M-1,\ldots,M-1,M,0)$.
Then the following are equivalent:
\begin{enumerate}
\item[(1)] $M\in\{d-1,d\}$;
\item[(2)] $\polytope{S}_{\b a,\b b}$ admits a unimodular triangulation;
\item[(3)] $\polytope{S}_{\b a,\b b}$ has the integer decomposition property.
\end{enumerate}
\end{corollary}

\begin{proof}
The implication $(1)\Rightarrow(2)$ is \Cref{thm_UT_2}, and
$(2)\Rightarrow(3)$ holds for every lattice polytope.  If $M<d-1$, then
\Cref{prop:non_idp_two_rows} shows that the simplex does not satisfy the integer decomposition property, proving
$(3)\Rightarrow(1)$ under the assumption $M\leq d$.
\end{proof}

\subsection{Ehrhart theoretic properties of  $\polytope{S}_{\b a, \b b}$}
We start with the computation of the $h^\ast$-polynomial for the specific class of two-row Hermite normal form simplices we are considering.

\begin{theorem}\label{thm:h_vector_twoR}
Let $N\in \N_{>1}$ and let $\polytope{S}_{\b a, \b b} $ be the two-row Hermite normal form simplex associated with  $\b a = (1, \dots, 1, N)$  and $\b b = (M-1, \ldots, M-1, M , 0)$.
Then,
\[
h^\ast_{\polytope{S}_{\b a,\b b}}(t)=\begin{cases}
     1 + N\sum_{j = 1}^{d-2}t^j + (N-1)t^{d-1}, &\qquad\text{ if } M = d - 1,\\
      1 + N\sum_{j = 1}^{d-1}t^j + (N-1)t^{d},  &\qquad\text{ if } M=d.
\end{cases}
\]
\end{theorem}
\begin{proof}
Let $\mathcal{T}$ be the unimodular triangulation of $\polytope{S}_{\b a, \b b}$ from \Cref{thm_UT_2}.
Assume first that $M=d$. 
Let $\polytope{F}\coloneqq\conv(\b v_0,\dots,\b v_{d-1})$. 
By the proof of \Cref{thm_UT_2}, the simplex $\polytope{S}_{\b a,\b b}$ satisfies Definition \ref{def:colinear_cone_triangulation} with base facet $\polytope{F}$, and the lattice points $\b p_1,\dots,\b p_N=\b v_d$ lying on a fixed line. 
Hence $m=N$ in \Cref{thm:hstar_collinear_cone}. 
The triangulation of $\polytope{F}$ induced by $\mathcal T$ is the stellar subdivision of $\polytope{F}$ with center $\b p_0$. 
Since $\mathcal T$ is unimodular, this induced triangulation is unimodular. 
Hence $h^\ast_{\polytope{F}}(t)=1+t+\cdots+t^{d-1}$. 
Therefore \Cref{thm:hstar_collinear_cone} gives
$$
h^\ast(t)=h^\ast_{\polytope{F}}(t)+(N-1)(t+t^2+\cdots+t^d)=1+N\sum_{j=1}^{d-1}t^j+(N-1)t^d.
$$

Now assume that $M=d-1$.
Let $\polytope{Q}\coloneqq\conv(\b v_1,\dots,\b v_d)$. Note that  $\polytope{S}_{\b a,\b b}=\conv(\b v_0,\polytope{Q})$ is a lattice pyramid over $\polytope{Q}$ and hence $h^\ast_{\polytope{S}_{\b a,\b b}}(t)=h^\ast_{\polytope{Q}}(t)$,
The same line of argument as in the previous case shows 

$$
h^\ast_{\polytope{Q}}(t)=h^\ast_{\polytope{H}}(t)+(N-1)(t+t^2+\cdots+t^{d-1})=1+N\sum_{j=1}^{d-2}t^j+(N-1)t^{d-1}
$$
which finishes the proof.
\end{proof}

Next we compute the local $h^\ast$-polynomial.

\begin{theorem}\label{thm:local_h_two_rows}
Let $N\in \N_{>1}$ and let $\polytope{S}_{\b a, \b b} $ be the two-row Hermite normal form simplex associated with  $\b a = (1, \dots, 1, N)$  and $\b b = (M-1, \ldots, M-1, M , 0)$.
Then,
$$
\ell^\ast_{\polytope{S}_{\b a,\b b}}(t)=
\begin{cases}
0, & \text{if } M=d-1,\\[0.2em]
(N-1)\displaystyle\sum_{i=1}^d t^i, &\text{if }M=d.
\end{cases}
$$
\end{theorem}

\begin{proof}
Assume first that \(M=d-1\), and set $\polytope Q\coloneqq\conv(\b v_1,\ldots,\b v_d)$. We have already seen in the proof of \Cref{thm:h_vector_twoR} that $\polytope{S}_{\b a,\b b}
=
\conv(\b v_0,\polytope Q)$ 
is a lattice pyramid with apex \(\b v_0\) and base \(\polytope Q\). It now follows directly from \cite[Proposition 3.5]{BorKreNill} that $\ell^\ast_{\polytope{S}_{\b a,\b b}}(t)=0$.

Now assume that \(M=d\). We will use \eqref{eq:local_h_S} to compute the local $h^\ast$-polynomial in this case. Let \(A\) be the homogenized vertex matrix. From a direct (easy) computation of $A^{-1}$ we immediately get that the fundamental parallelepiped group $\Gamma\coloneqq \Z^{d+1}A^{-1}/\Z^{d+1}$ is generated by 
 \[\b g
=
\left(
\underbrace{\frac1d,\ldots,\frac1d}_{d\text{ coordinates}},0
\right) \quad \text{and}\quad \b h
=
\left(
\underbrace{-\frac1{dN},\ldots,-\frac1{dN}}_{d\text{ coordinates}},
\frac1N
\right).
\]
Since $\#\Gamma=|\det(A)|=dN$, it follows that every element of $\Gamma$ is of the form $\b x_{q,m}
=
\operatorname{frac}(q\b g+m\b h)$, where $0\leq q\leq d-1$ and $0\leq m\leq N-1$ (and this representation is unique).

If \(m=0\), then \(\b x_{q,0}=0\) is zero and hence, by definition,  these do not contribute to the local \(h^\ast\)-polynomial (cf. \eqref{eq:local_h_S}).

Now let \(1\leq m\leq N-1\). If \(q=0\), then $\b x_{0,m}=\left(
1-\frac{m}{dN},\ldots,1-\frac{m}{dN},\frac{m}{N}\right)$,
so all its coordinates are nonzero and 
\[
\operatorname{age}(\b x_{0,m})
=
d\left(1-\frac{m}{dN}\right)+\frac{m}{N}
=
d.
\]

If \(1\leq q\leq d-1\), then $0<
\frac{q}{d}-\frac{m}{dN}
<1$, and hence $\b x_{q,m}
=
\left(
\frac{q}{d}-\frac{m}{dN},\ldots,
\frac{q}{d}-\frac{m}{dN},
\frac{m}{N}
\right)$.
Again, all coordinates are nonzero, and
\[
\operatorname{age}(\b x_{q,m})
=
d\left(\frac{q}{d}-\frac{m}{dN}\right)+\frac{m}{N}
=
q.
\]
Combining these computations, we get 
\[
\ell^\ast_{\polytope{S}_{\b a,\b b}}(t)
=
\sum_{m=1}^{N-1}\sum_{j=1}^d t^j
=
(N-1)\sum_{j=1}^d t^j,
\]
which completes the proof.
\end{proof}
\begin{theorem}
Let $N\in \N_{>1}$ and let $\polytope{S}_{\b a, \b b}$ be the Hermite normal form simplex associated with $\b a = (1, \dots, 1, N)$ and $\b b = (M-1, \ldots, M-1,M,0)$.  
If $M\in\{ d - 1,d\}$, then $\polytope{S}_{\b a, \b b}$ is Ehrhart positive.
\end{theorem}

\begin{proof} 
First, suppose that $M=d$. 
Comparing \Cref{thm:h_vector_twoR} and \Cref{thm:h_vector_OneR} we see that the $ h^\ast$-polynomial of $\polytope{S}_{\b a,\b b}$ is of the same form as for $\polytope{S}_{\b a}$. Since both simplices have the same dimension, it follows from  \Cref{thm:Ehrhar_one_row} that $\polytope{S}_{\b a,\b b}$ is Ehrhart positive.

Now suppose that $M = d - 1$. Using \Cref{thm:h_vector_twoR}, the usual transformation from the $h^\ast$-vector to the Ehrhart polynomial and \Cref{lem:Sum_binomial}, an easy computation shows
\begin{align*}
\mathcal{L}_{\polytope{S}_{\b a, \b b}
}(t)= \binom{t + d}{d} - \binom{t + 1}{d} + N \left( \binom{t + d}{d + 1} - \binom{t + 1}{d + 1} \right),
\end{align*}
The claim follows from \Cref{coro:Binomial_difference_2}.
\end{proof}

\begin{remark}
When $M=d$, the Ehrhart polynomial has the same form as in the one-row case
$N=dk$, with the parameter $k$ replaced by the present parameter $N$.
Consequently, \Cref{thm:explicit_non_unimodality_bound} shows that, for
$d\geq3$ and $N\geq K_D(d)$, the Ehrhart polynomial of
$\polytope{S}_{\b a,\b b}$ is not unimodal.
For $M=d-1$, computer experiments suggest that the Ehrhart polynomial is unimodal. 
At present, we do not know a proof of this fact.
\end{remark}

\begin{question}\label{question:unimodal_two_rows}
Let $\polytope{S}_{\b a,\b b}$ be the two-row Hermite normal form simplex associated with $\b a=(1,\ldots,1,N)$ and $\b b=(d-2,\ldots,d-2,d-1,0)$. 
Is the Ehrhart polynomial of $\polytope{S}_{\b a,\b b}$ unimodal for all $N\in\N_{>1}$?
\end{question}

The result above highlights a contrast with the one-row case. 
When $\b a = (1, \dots, 1, N)$ and $d>2$, the associated one-row Hermite normal form simplex fails to satisfy the integer decomposition property (see \Cref{prop:NO_IPD_1N}).
In contrast, the addition of a second row $\b b = (M - 1, \dots, M - 1, M, 0)$ with $M = d - 1$ or $M = d$ not only ensures that the simplex satisfies the integer decomposition property, but also guarantees that all Ehrhart coefficients are strictly positive.

\bibliographystyle{plain}
\bibliography{Bib.bib}

\end{document}